\documentclass[12pt]{amsart}
\usepackage{amssymb,amsmath,amsfonts, amsthm}
\usepackage{graphicx}
\usepackage{fullpage}

\newcommand{\R}{\mathbb{R}}
\newcommand{\N}{\mathbb{N}}
\newcommand{\Z}{\mathbb{Z}}
\newcommand{\Q}{\mathbb{Q}}
\newcommand{\C}{\mathbb{C}}
\newcommand{\bbH}{\mathbb{H}}
\newcommand{\K}{\mathbb{K}}

\newcommand{\Diag}{\mathrm{Diag}}
\newcommand{\tr}{\mathrm{tr}}

\newcommand{\Tr}{\mathrm{Tr}}
\newcommand{\Nrd}{\mathrm{Nrd}}
\newcommand{\GL}{\mathrm{GL}}
\newcommand{\SL}{\mathrm{SL}}
\newcommand{\PSL}{\mathrm{PSL}}

\newcommand{\Ha}{\mathrm{\textbf{H}}}
\newcommand{\RP}{\mathbb{RP}}
\newcommand{\Lim}{\mathcal{L}}

\newtheorem{theorem}{Theorem}
\newtheorem{proposition}[theorem]{Proposition}
\newtheorem{corollary}[theorem]{Corollary}
\newtheorem{lemma}[theorem]{Lemma}

\newtheorem*{thmA}{Theorem A}
\newtheorem*{thmB}{Theorem B}
\newtheorem*{thmC}{Theorem C}

\theoremstyle{definition}

\theoremstyle{remark}
\newtheorem*{rem}{Remark}

\title{THE LIMIT SET OF SUBGROUPS OF ARITHMETIC GROUPS in $\PSL(2,\C)^q\times\PSL(2,\R)^r$}
\author{Slavyana Geninska}

\begin{document}

\maketitle
\begin{abstract}
While lattices in semi-simple Lie groups are studied very well, only little is known about discrete subgroups of infinite covolume. The main class of examples are Schottky groups. Here we investigate some new examples.

We consider subgroups $\Gamma$ of arithmetic groups in $\PSL(2,\C)^q\times \PSL(2,\R)^r$ with $q+r\geq 2$ and their limit set. We prove that the projective limit set of a nonelementary finitely generated $\Gamma$ consists of exactly one point if and only if one and hence all projections of $\Gamma$ to the simple factors of $\PSL(2,\C)^q\times \PSL(2,\R)^r$ are subgroups of arithmetic Fuchsian or Kleinian groups. Furthermore, we study the topology of the whole limit set of $\Gamma$. In particular, we give a necessary and sufficient condition for the limit set to be homeomorphic to a circle.
\end{abstract}

\pagestyle{plain}
\setcounter{page}{1}
\setcounter{tocdepth}{1}

\setcounter{section}{-1}
\section{Introduction}
Arithmetic subgroups of semi-simple Lie groups have been and still are a major subject of study. They are examples of lattices, i.e. discrete subgroups of finite covolume. Margulis' Arithmeticity Theorem states that for groups with $\R$-rank greater than or equal to 2, the only lattices are the arithmetic ones (see Margulis~\cite{gM91}, Theorem A, p.298). 

The situation is very different for the simple Lie groups $\PSL(2,\R)$ and $\PSL(2,\C)$. There, the arithmetic lattices represent a minority among all lattices, i.e. the cofinite Fuchsian and Kleinian groups. Nevertheless, they provide important examples since one can get the general form of their elements quite explicitly and not only in terms of generators.

The so called semi-arithmetic Fuchsian groups constitute a specific class of Fuchsian groups which can be embedded up to commensurability in arithmetic subgroups of $\PSL(2,\R)^r$ (see Schmutz Schaller and Wolfart~\cite{pS00}). These embeddings are of infinite covolume in $\PSL(2,\R)^r$. A trivial example is the group $\PSL(2,\Z)$ that can be embedded diagonally in any Hilbert modular group. Further examples are the other arithmetic Fuchsian groups and the triangle Fuchsian groups. It is a general question if certain classes of Fuchsian groups can be characterized by geometric means. There is a classical  characterization of arithmetic Fuchsian groups due to Takeuchi which is based on number theoretical properties of  their trace sets  \cite{kT75}. A further characterization in this direction is given in \cite{sG08}.

While lattices are studied very well, only little is known about discrete subgroups of infinite covolume of semi-simple Lie groups (e.g P. Albuquerque~\cite{pA99}, Leuzinger~\cite{eL08}, Link~\cite{gL02}, Quint~\cite{jQ02,jQ05}). The main class of examples are Schottky groups. There are also some rigidity results for subgroups of the product of two simple Lie groups of real rank 1, see e.g. Burger~\cite{mB93} and Dal'bo and Kim~\cite{fD00}.  

The chief goal of this article is to provide and investigate further examples of infinite covolume subgroups. Namely, we consider subgroups of irreducible lattices in $\PSL(2,\C)^q\times\PSL(2,\R)^r$ with the property that the projection to one of the factors is a Fuchsian (or a Kleinian) group. We are in particular interested in those groups whose projection to one of the factors is (a subgroup of) an arithmetic Fuchsian (or Kleinian) group. These are exactly the nonelementary groups with the ``smallest" possible limit set. It is somewhat surprising that one can get information about the arithmetic nature of a group from its limit set.



We recall the definition of the limit set. The symmetric space corresponding to $\PSL(2,\C)^q\times\PSL(2,\R)^r$ is a product of 2- and 3-dimensional real hyperbolic spaces $(\bbH^3)^q\times(\bbH^2)^r$. The geometric boundary of $(\bbH^3)^q\times(\bbH^2)^r$ is the set of equivalence classes of asymptotic geodesic rays. Adding the geometric boundary to $(\bbH^3)^q\times(\bbH^2)^r$ yields a compactification of $(\bbH^3)^q\times(\bbH^2)^r$ compatible with the action of $\PSL(2,\C)^q\times\PSL(2,\R)^r$. The limit set $\Lim_\Gamma$ of a discrete subgroup $\Gamma$ of $\PSL(2,\C)^q\times\PSL(2,\R)^r$ is the part of the orbit closure $\overline{\Gamma(x)}$ in the geometric boundary where $x$ is an arbitrary point in $(\bbH^3)^q\times(\bbH^2)^r$. Hence the limit set gives us information about the group by looking at its action ``far away".

This article is organized as follows. In Section 1 we have compiled some basic facts about Fuchsian and Kleinian groups and especially about Schottky groups. We prove in particular a criterion for a Schottky subgroup of $\PSL(2,\C)$ to be Zariski dense over $\R$.

Section 2 provides a detailed description of the geometric boundary of $(\bbH^3)^q\times(\bbH^2)^r$. We also introduce the notion of the limit set and state a structure theorem for the regular limit set $\Lim_\Gamma^{reg}$ of discrete nonelementary groups $\Gamma$ due to Link: $\Lim_\Gamma^{reg}$ is the product of the Furstenberg limit set $F_\Gamma$ and the projective limit set $P_\Gamma$. We also give a criterion for a subgroup of $\PSL(2,\C)^q\times\PSL(2,\R)^r$ to be nonelementary.

Section 3 is devoted to the study of nonelementary groups. We show that the regular limit set of a nonelementary group is not empty. The main result of the section is the following theorem, which is Theorem~\ref{T:ConvexAlg} in the text.

\begin{thmA}
Let $\Gamma$ be a nonelementary subgroup of an irreducible arithmetic group in $\PSL(2,\C)^q\times\PSL(2,\R)^r$ with $q + r \geq 2$. Then $P_\Gamma$ is convex and the closure of $P_\Gamma$ in $\RP^{q + r-1}$  is equal to the limit cone of $\Gamma$ and in particular the limit cone of $\Gamma$ is convex.
\end{thmA}

This is a result similar to a theorem of Benoist in Section 1.2 in \cite{yB97}. Note however that while the interior of the limit cone is always nonempty for Zariski dense groups, it can be empty for groups that are just nonelementary.

In Section 4 we describe the irreducible arithmetic subgroups of the group $\PSL(2,\C)^q\times\PSL(2,\R)^r$ using quaternion algebras. 

In Section 5 we consider subgroups (of infinite index) of irreducible arithmetic groups in $\PSL(2,\C)^q\times\PSL(2,\R)^r$. We start by giving the example of Hecke groups embedded in Hilbert modular groups. In \S5.3 we determine the groups with the smallest possible limit set. The main  result is a compilation of Theorem~\ref{T:MainCharactAlg} and Corollary~\ref{C:WidelyComm}.
\begin{thmB}
Let $\Gamma$ be as in Theorem A and in addition be finitely generated. Then the projective limit set $P_{\Gamma}$ consists of exactly one point if and only if $\Gamma$ is a conjugate by an element in $\GL(2,\C)^q\times\GL(2,\R)^r$ of a subgroup of
$$ \Diag(S):=\{(\sigma_1(s),\ldots,\sigma_{q+r}(s))\mid s\in S\},$$
where $S$ is an arithmetic Fuchsian or Kleinian group and, for each $i=1,\ldots,q+r$, $\sigma_i$ denotes either the identity or the complex conjugation.
\end{thmB}

In particular, Theorem~\ref{T:MainCharactAlg} shows that the projective limit set $P_{\Gamma}$ consists of exactly one point if and only if $p_j(\Gamma)$ is contained in an arithmetic Fuchsian or Kleinian group for one and hence all $j\in\{1,\ldots,q+r\}$.

The main ingredients of the proof are the characterization of cofinite arithmetic Fuchsian groups by Takeuchi \cite{kT75} (and the analogous characterization for cofinite Kleinian groups given by Maclachlan and Reid in \cite{cM03}), the criterion for Zariski density of Dal'Bo and Kim \cite{fD00} and a theorem of Benoist \cite{yB97} stating that for Zariski dense subgroups of $\PSL(2,\C)^q\times \PSL(2,\R)^r$ the projective limit cone has nonempty interior. 

There are some differences depending on whether $q=0$ or $r=0$ or $qr\neq 0$. The group $S$ and hence $p_j(\Gamma)$ can be an arithmetic Kleinian group only if $r=0$. And if $q=0$, we can require that only one of the projections of $\Gamma$ is nonelementary. This is due to the following fact. If $\Delta$ is a subgroup of an irreducible arithmetic group in $\PSL(2,\R)^r$ and $\Gamma$ is a subgroup of $\Delta$ such that its projection to one factor is nonelementary, then $\Gamma$ is nonelementary (Lemma~\ref{L:Nonelementary}). This is no longer true in the general case as the example in the remark after Corollary~\ref{C:MainCharactAlgAlg} shows.

We can also avoid the assumption that $\Gamma$ is nonelementary by using the limit cone as  Theorem~\ref{T:CharactLimitCone} shows.

We then consider the full limit set (instead of $P_\Gamma$ only) and determine the nonelementary groups with the smallest one. The following theorem is a compilation of Theorem~\ref{T:TopolCharactAlgAlg} and Corollary~\ref{C:TopolCharactAlgAlg}.

\begin{thmC}
Let $\Gamma$ be a finitely generated subgroup of an irreducible arithmetic group in $\PSL(2,\C)^q\times\PSL(2,\R)^r$ with $q + r \geq 2$ and $r\neq 0$ such that $p_{j}(\Gamma)$ is nonelementary for some $j=1,\ldots,q+r$. Then $\mathcal{L}_\Gamma$ is embedded homeomorphically in a circle if and only if $p_{j}(\Gamma)$ is contained in an arithmetic Fuchsian group.

Furthermore $\mathcal{L}_\Gamma$ is homeomorphic to a circle if and only if $p_{j}(\Gamma)$ is a cofinite arithmetic Fuchsian group.
\end{thmC}

In the case $r=0$, we determine for which $\Gamma$ the limit set is diffeomorphic to the 2-sphere (Theorem~\ref{T:TopolCharactC}) and for which it is homeomorphic to a circle (Theorem~\ref{T:TopolCharactCFuchs}).

\bigskip

\textit{Acknowledgments.} I would like to thank my advisor Enrico Leuzinger for the many stimulating discussions. Furthermore, I would like to thank Fran\c{c}oise Dal'Bo and Gabriele Link for pointing out and explaining me some of their results, to Corentin Boissy for helpful discussions and to Alan Reid for comments on this article.

\section{Schottky Groups in $\PSL(2,\C)$ and $\PSL(2,\R)$}
\setcounter{theorem}{0}
In this section we provide some basic facts and notations that are needed in the rest of this paper. We will use the notation introduces in Chapter 2 in the book of Maclachlan and Reid \cite{cM03}.

We will change freely between matrices in $\SL(2,\C)$ and their action as fractional linear transformations, namely as elements in $\PSL(2,\C)$.

For $g = \begin{bmatrix} a&b\\c&d \end{bmatrix} \in \PSL(2,\C)$ we set $\tr(g):=\pm(a+d)$, where the sign is chosen so that $\tr(g)=re^{i\theta}$ with $r\geq 0$ and $\theta\in[0,\pi)$. Note that for $g\in \PSL(2,\R)$ we have $\tr(g)=|a+d|$. 

For a subgroup $\Gamma$ of $\PSL(2,\C)$ we call 
\[ \Tr(\Gamma) = \{\tr(g) \mid g \in \Gamma \} \]
the \textit{trace set of} $\Gamma$. 

The \textit{translation length} of a loxodromic $g$ is the distance between a point $x$ on the geodesic fixed by $g$ and its image $g(x)$ under $g$. If $g$ is elliptic, parabolic or the identity, we define $\ell(g):=0$.

The following notion of ``smallness" for subgroups $\Gamma$ of $\PSL(2,\C)$ is important in the subsequent discussion. The group $\Gamma$ is \textit{elementary} if it has a finite orbit in its action on $\bbH^3\cup\C\cup\{\infty\}$. Otherwise it is said to be \textit{nonelementary}. Every nonelementary subgroup of $\PSL(2,\C)$ contains infinitely many loxodromic elements, no two of which have a common fixed point (see Theorem 5.1.3 in the book of Beardon \cite{aB95}).

A Schottky group is a finitely generated free subgroup of $\PSL(2,\C)$ that contains only loxodromic isometries except for the identity. We will mainly deal with two-generated Schottky groups.

\begin{lemma}
\label{L:SchottkyLox}
For each two loxodromic isometries without common fixed points, we can find powers of them that generate a Schottky group. This means that every nonelementary subgroup of $\PSL(2,\C)$ has a subgroup that is a Schottky group.
\end{lemma}

A proof of this lemma can be found in \cite{sG09}.

A Schottky group contains isometries without common fixed points because it is nonelementary.

Everything above is also true for Fuchsian Schottky groups, i.e for subgroups of $\PSL(2,\C)$ that after conjugation become subgroups of $\PSL(2,\R)$.

Let $\K$ be either $\C$ or $\R$. The following is shown by Cornelissen and Marcolli in \cite{gC08}, Lemma~3.4.

\begin{lemma}[\cite{gC08}]
\label{L:ZariskiK}
A Schottky group is Zariski dense over $\K$ in $\PSL(2,\K)$.
\end{lemma}

Since every nonelementary subgroup of $\PSL(2,\K)$ contains a Schottky group, we have the following
\begin{corollary}
A nonelementary subgroup of $\PSL(2,\K)$ is Zariski dense over $\K$ in $\PSL(2,\K)$.
\end{corollary}

The next question is when $\Gamma$ is Zariski dense over $\R$. By Corollary 3.2.5 in the book of Maclachlan and Reid \cite{cM03}, if $\Tr(\Gamma)$ is a subset of $\R$, then $\Gamma$ is conjugate to a subgroup of $\PSL(2,\R)$. Then the Zariski closure of $\Gamma$ over $\R$ is a conjugate of $\PSL(2,\R)$. For the case when $\Tr(\Gamma)$ is not a subset of $\R$ we have the following lemma.

\begin{lemma}
\label{L:ZariskiC}
Let $\Gamma$ be a Schottky group such that $\Tr(\Gamma)$ is not a subset of $\R$. Then $\Gamma$ is Zariski dense in $\PSL(2,\C)$ over $\R$.
\end{lemma}
\begin{proof}
Let $\hat{\Gamma}$ be the Zariski closure of $\Gamma$ over $\R$. Then $\hat{\Gamma}$ is an algebraic group and hence a Lie subgroup of $\PSL(2,\C)$. Since $\PSL(2,\C)$ is connected, it is enough to show that the dimension of $\hat{\Gamma}$ is equal to the dimension of $\PSL(2,\C)$ over $\R$ in oder to conclude that $\Gamma$ is Zariski dense in $\PSL(2,\C)$ over $\R$.

We will show that the Lie algebra of $\hat{\Gamma}$ over $\R$ is equal to $\mathfrak{sl}(2,\C)$, which is the Lie algebra of $\PSL(2,\C)$ considered as a real Lie group.

First we consider $\Gamma$ and $\hat{\Gamma}$ as subgroups of $\SL(2,\C)$. Then the exponential map from the Lie algebra of $\hat{\Gamma}$ to $\hat{\Gamma}$ is given by the matrix exponential map.

After conjugation, since $\Tr(\Gamma)$ is not a subset of $\R$, we can assume that $\Gamma$ contains a loxodromic element $T_1 = \begin{bmatrix}e^x & 0 \\0 & e^{-x}\end{bmatrix}$ with $x \in \C\backslash\R$ and and hence $\tr(T_1)\notin \R$. There is also an isometry $T_2$ which does not have common fixed points with $T_1$ and hence is equal to $T_2 = \begin{bmatrix}a' & b' \\c' & d'\end{bmatrix}$ with $b'\neq 0$ and $c'\neq 0$ (otherwise $0$ or $\infty$ would be a common fixed point). Therefore their preimages in $\mathfrak{sl}(2,\C)$ under the exponential map are $t_1=\begin{bmatrix}x & 0 \\0 & -x\end{bmatrix}$ and $t_2 = \begin{bmatrix}a & b \\c & -a\end{bmatrix}$ where $b \neq 0$ and $c \neq 0$. Then $\mathfrak{sl}(2,\C)$ contains
$$\begin{array}{cc}
t_3 = [t_1,t_2] = \begin{bmatrix}0 & 2bx \\-2cx & 0 \end{bmatrix}, &
t_4 = [t_1,t_3] = \begin{bmatrix}0 & 4bx^2 \\4cx^2 & 0 \end{bmatrix}, \\
t_5 = [t_1,t_4] = \begin{bmatrix}0 & 8bx^3 \\-8cx^3 & 0 \end{bmatrix}, &
t_6 = [t_1,t_5] = \begin{bmatrix}0 & 16bx^4 \\16cx^4 & 0 \end{bmatrix}.\\ 
\end{array}$$
Since $\tr(T_1) = e^x + e^{-x}$ is not real, then $x$ is not only not real but also cannot be a multiple of $i$. This means that $x$ and $x^3$ are linearly independent over $\R$. Hence the linear span of $t_3$ and $t_5$ over $\R$ is $\{\begin{bmatrix}0 & bz \\-cz & 0 \end{bmatrix}\mid z \in \C \}$. 

Analogously, the linear span of $t_4$ and $t_6$ over $\R$ is $\{\begin{bmatrix}0 & bz \\cz & 0 \end{bmatrix}\mid z \in \C \}$.

Hence $t_3$, $t_4$, $t_5$ and $t_6$ span the 4-dimensional real vector subspace of the Lie algebra of $\hat{\Gamma}$
$$U := \{\begin{bmatrix}0 & z_1 \\z_2 & 0 \end{bmatrix}\mid z_1,z_2 \in \C \}.$$
By taking the commutator of the elements of $U$ with $\begin{bmatrix}0 & 0 \\1 & 0 \end{bmatrix}$, which is also an element in $U$, we get that the Lie algebra of $\hat{\Gamma}$ contains
$$V := \{\begin{bmatrix}z_1 & 0 \\0 & -z_1 \end{bmatrix}\mid z_1 \in \C \}.$$
The span of $U$ and $V$ is the 6-dimensional real vector space $\mathfrak{sl}(2,\C)$. Hence the Lie algebra of $\hat{\Gamma}$ is $\mathfrak{sl}(2,\C)$.

This means that if we consider $\Gamma$ and $\hat{\Gamma}$ as subgroups of $\PSL(2,\C)$, the Lie algebra of $\hat{\Gamma}$ is still $\mathfrak{sl}(2,\C)$.
\end{proof}

Since every nonelementary subgroup of $\PSL(2,\C)$ contains a Schottky group, we have the following
\begin{corollary}
\label{C:Zariski}
(i) A nonelementary subgroup $\Gamma$ of $\PSL(2,\C)$ is Zariski dense over $\R$ in $\PSL(2,\C)$ if and only if $\Tr(\Gamma)$ is not a subset of $\R$.

(ii) The Zariski closure over $\R$ of a nonelementary subgroup of $\PSL(2,\C)$ with real traces is a conjugate of $\PSL(2,\R)$.
\end{corollary}

\section{Products of hyperbolic planes and 3-spaces}
\setcounter{theorem}{0}
Let $q$ and $r$ be two nonnegative integers such that $q+r > 0$. We consider the product $(\bbH^3)^q\times(\bbH^2)^r$ of $q$ which is the Cartesian product of $q$ upper half 3-spaces and $r$ upper half planes with the product metric 
and we denote by $d$ the corresponding distance function. The Riemannian manifold $(\bbH^3)^q\times(\bbH^2)^r$ is a symmetric space of rank $q + r$.

In the next sections we will define the geometric boundary of $(\bbH^3)^q\times(\bbH^2)^r$ and the limit set of a group acting on $(\bbH^3)^q\times(\bbH^2)^r$ by isometries.

\subsection{The geometric boundary of $(\bbH^3)^q\times(\bbH^2)^r$}
For $i = 1,\ldots,q$, we denote by $p_i: (\bbH^3)^q\times(\bbH^2)^r \rightarrow \bbH^3$, $(z_1,...,z_{q+r}) \mapsto z_i$ the $i$-th projection of $(\bbH^3)^q\times(\bbH^2)^r$ into $\bbH^3$, and for $i=q+1,\ldots,q+r$, we denote by $p_i: (\bbH^3)^q\times(\bbH^2)^r \rightarrow \bbH^2$, $(z_1,...,z_{q+r}) \mapsto z_i$ the $i$-th projection of $(\bbH^3)^q\times(\bbH^2)^r$ into $\bbH^2$. Let $\gamma: [0, \infty )\rightarrow (\bbH^3)^q\times(\bbH^2)^r$ be a curve in $(\bbH^3)^q\times(\bbH^2)^r$. Then $\gamma$ is a geodesic ray in $(\bbH^3)^q\times(\bbH^2)^r$ if and only if $p_i\circ \gamma$ is a geodesic ray or a point in $\bbH^3$ for each $i = 1, \ldots , q$ and a geodesic ray or a point in $\bbH^2$ for each $i = q+1, \ldots , q+r$. A geodesic $\gamma$ is \textit{regular} if $p_i\circ \gamma$ is a nonconstant geodesic in $\bbH^3$ for each $i = 1, \ldots , q$ and a nonconstant geodesic in $\bbH^2$ for each $i = q+1, \ldots , q+r$.

Two unit speed geodesic rays $\gamma$ and $\delta$ in $(\bbH^3)^q\times(\bbH^2)^r$ are said to be asymptotic if there exists a positive number $c$ such that $d(\gamma(t), \delta(t)) \leq c$ for all $t \geq 0$. This is an equivalence relation on the unit speed geodesic rays of $(\bbH^3)^q\times(\bbH^2)^r$. For any unit speed geodesic $\gamma$ of $(\bbH^3)^q\times(\bbH^2)^r$ we denote by $\gamma(+\infty)$ the equivalence class of its positive ray. 

We denote by $\partial ((\bbH^3)^q\times(\bbH^2)^r)$ the set of all equivalence classes of unit speed geodesic rays of $(\bbH^3)^q\times(\bbH^2)^r$. We call $\partial ((\bbH^3)^q\times(\bbH^2)^r)$ the \textit{geometric boundary} of $(\bbH^3)^q\times(\bbH^2)^r$. The \textit{regular boundary} $\partial ((\bbH^3)^q\times(\bbH^2)^r)_{reg}$ of $(\bbH^3)^q\times(\bbH^2)^r$ consists of the equivalence classes of regular geodesics. 

The geometric boundary $\partial ((\bbH^3)^q\times(\bbH^2)^r)$ with the cone topology is homeomorphic to the unit tangent sphere of a point in $(\bbH^3)^q\times(\bbH^2)^r$ (see Eberlein \cite{pE96}, 1.7). (For example $\partial \bbH^2$ is homeomorphic to $S^1$.) The homeomorphism is given by the fact that for each point $x_0$ and each unit speed geodesic ray $\gamma$ in $(\bbH^3)^q\times(\bbH^2)^r$ there exists a unique unit speed geodesic ray $\delta$ with $\delta(0) = x_0$ which is asymptotic to $\gamma$.

The group $\PSL(2,\C)^q\times\PSL(2,\R)^r$ acts on $(\bbH^3)^q\times(\bbH^2)^r$ by isometries in the following way. For $g = (g_1,\ldots, g_{q+r}) \in \PSL(2,\C)^q\times\PSL(2,\R)^r$
$$ g:(\bbH^3)^q\times(\bbH^2)^r \rightarrow (\bbH^3)^q\times(\bbH^2)^r, \quad (z_1, \ldots, z_{q+r}) \mapsto (g_1 z_1, \ldots, g_{q+r} z_{q+r}), $$
where $z_i \mapsto g_i z_i$ is the usual action given by linear fractional transformation, $i = 1, \ldots, {q+r}$.

The action of $\PSL(2,\C)^q\times\PSL(2,\R)^r$ can be extended naturally to $\partial ((\bbH^3)^q\times(\bbH^2)^r)$. Let $g$ be in $\PSL(2,\C)^q\times\PSL(2,\R)^r$ and $\xi$ be a point in the boundary $\partial ((\bbH^3)^q\times(\bbH^2)^r)$. If $\gamma$ is a representative of $\xi$, then $g(\xi)$ is the equivalence class of the geodesic ray $g\circ \gamma$.

We call $g$ \textit{elliptic} if all $g_i$ are elliptic isometries, \textit{parabolic} if all $g_i$ are parabolic isometries, \textit{loxodromic} if all $g_i$ are loxodromic isometries and \textit{hyperbolic} if all $g_i$ are hyperbolic isometries. In all the other cases we call $g$ \textit{mixed}. 

If at least one $\ell(g_i)$ is different from zero, then we define the \textit{translation direction} of $g$ as $L(g):= (\ell(g_1): \ldots: \ell(g_{q+r})) \in \RP^{q+r-1}$.

\subsection{Decomposition of the geometric boundary of $(\bbH^3)^q\times(\bbH^2)^r$}
In this section we show a natural decomposition of the geometric boundary of $(\bbH^3)^q\times(\bbH^2)^r$ and in particular of its regular part. This is a special case of a general construction for a large class of symmetric spaces (see e.g. Leuzinger~\cite{eL92} and Link~\cite{gL02}). This decomposition plays a main role in this article.

Let $x=(x_1,\ldots,x_{q+r})$ be a point in $(\bbH^3)^q\times(\bbH^2)^r$. We consider the \textit{Weyl chambers} with vertex $x$ in $(\bbH^3)^q\times(\bbH^2)^r$ given by the product of the images of the geodesics $\delta_i:[0,\infty)\rightarrow \bbH^3$ with $\delta_i(0)=x_i$ for $i = 1,\ldots,q$ and  $\delta_i:[0,\infty)\rightarrow \bbH^2$ with $\delta_i(0)=x_i$ for $i = q+1,\ldots,q+r$. 

The isotropy group in $\PSL(2,\C)^q\times\PSL(2,\R)^r$ of $x$ is $PSU(2)^q \times PSO(2)^r$. It acts transitively but not simply transitively on the Weyl chambers with vertex $x$ because a fixed Weyl chamber with vertex $x$ is left unchanged by a group isomorphic to $ PSO(2)^q \times \{id\}^r$. Hence the group acting simply transitively on the Weyl chambers with vertex $x$ is $(PSU(2)/PSO(2))^q \times PSO(2)^r$. 

Let $W$ be a Weyl chamber with vertex $x$. In $W$, two unit speed geodesics $\gamma(t) = (\gamma_1(t),\ldots,\gamma_{q+r}(t))$ and $\tilde{\gamma} = (\tilde{\gamma}_1(t),\ldots,\tilde{\gamma}_{q+r}(t))$ are different if and only if the corresponding projective points $$\left(d_H(\gamma_1(0),\gamma_1(1)):\ldots:d_H(\gamma_{q+r}(0),\gamma_{q+r}(1))\right) \text{ and}$$ $$(d_H(\tilde{\gamma}_1(0),\tilde{\gamma}_1(1)):\ldots:d_H(\tilde{\gamma}_{q+r}(0),\tilde{\gamma}_{q+r}(1)))$$ 
are different. Here $d_H$ denotes the hyperbolic distance in $\bbH^3$ and $\bbH^2$. The point in $\RP^{q+r-1}$ given by $\left(d_H(\gamma_1(0),\gamma_1(1)):\ldots:d_H(\gamma_{q+r}(0),\gamma_2(1))\right)$ is a direction in the Weyl chamber and it is the same as $(\left\|v_1\right\|:\ldots:\left\|v_{q+r}\right\|)$, where $v = (v_1,\ldots,v_{q+r}):= \gamma'(0)$ is the unit tangent vector of $\gamma$ in $0$. 

In other words we can extend the action of $Iso_x$ to the tangent space at $x$ in $(\bbH^3)^q\times(\bbH^2)^r$. Then $Iso_x$ maps a unit tangent vector at $x$ onto a unit tangent vector at $x$.
Let $v$ be a unit tangent vector at $x$ in $(\bbH^3)^q\times(\bbH^2)^r$. We denote by $v_i$ the $i$-th projection of $v$ on the tangent spaces at $x_i$, $i=1,\ldots,q+r$. Then all the vectors $w$ in the orbit of $v$ under $Iso_x$ have $\left\|w_i\right\|=\left\|v_i\right\|$. 

Let $v$ be a vector in the unit tangent sphere at $x$ in$(\bbH^3)^q\times(\bbH^2)^r$. If $v$ is tangent to a regular geodesic, then the orbit of $v$ is homeomorphic to $(S^2)^q\times (S^1)^r \cong \left( \partial \bbH^3 \right)^q\times \left( \partial \bbH^2 \right)^r $ because $ \partial \bbH^3 \cong S^2$ and $ \partial \bbH^2 \cong S^1$. The orbit of $v$ under the group $(PSU(2)/PSO(2))^q \times PSO(2)^r$ consists of all unit tangent vectors $w$ at $x$ such that $\left\|w_i\right\|=\left\|v_i\right\|$ for $i=1,\ldots,q+r$.

The \textit{regular boundary} $\partial ((\bbH^3)^q\times(\bbH^2)^r)_{reg}$ of $(\bbH^3)^q\times(\bbH^2)^r$ consists of the equivalence classes of regular geodesics. Hence it is identified with $\left( \partial \bbH^3 \right)^q\times \left( \partial \bbH^2 \right)^r \times \RP^{q+r-1}_+$ where 
$$ \RP^{q+r-1}_+ := \left\{(w_1:\ldots:w_{q+r}) \in \RP^{q+r-1} \mid w_1 > 0, \ldots, w_{q+r} > 0 \right\}. $$
Here $w_1,..,w_{q+r}$ can be thought as the norms of the projections of the regular unit tangent vectors on the simple factors of $(\bbH^3)^q\times(\bbH^2)^r$.

$\left( \partial \bbH^3 \right)^q\times \left( \partial \bbH^2 \right)^r$ is called the \textit{Furstenberg boundary} of $(\bbH^3)^q\times(\bbH^2)^r$. 

We note that the decomposition of the boundary into orbits under the group $Iso_x$ is independent of the point $x$.

\subsection{The limit set of a group}
Let $x$ be a point and $\{x_n\}_{n\in \N}$ a sequence of points in $(\bbH^3)^q\times(\bbH^2)^r$. We say that $\{x_n\}_{n\in \N}$ converges to a point $\xi \in \partial\left((\bbH^3)^q\times(\bbH^2)^r\right)$ if $\{x_n\}_{n\in \N}$ is discrete in $(\bbH^3)^q\times(\bbH^2)^r$ and the sequence of geodesic rays starting at $x$ and going through $x_n$ converges towards $\xi$ in the cone topology. With this topology, $(\bbH^3)^q\times(\bbH^2)^r \cup \partial\left((\bbH^3)^q\times(\bbH^2)^r\right)$ is a compactification of $(\bbH^3)^q\times(\bbH^2)^r$.

Let $\Gamma$ be a subgroup of $\PSL(2,\C)^q\times\PSL(2,\R)^r$. We denote by $\Gamma(x)$ the orbit of $x$ under $\Gamma$ and by $\overline{\Gamma(x)}$ - its closure. The \textit{limit set} of $\Gamma$ is $\mathcal{L}_\Gamma:= \overline{\Gamma(x)}\cap \partial\left((\bbH^3)^q\times(\bbH^2)^r\right)$. The limit set is independent of the choice of the point $x$ in $(\bbH^3)^q\times(\bbH^2)^r$. The \textit{regular limit set} is $\mathcal{L}_\Gamma^{reg}:=\mathcal{L}_\Gamma \cap  \partial\left((\bbH^3)^q\times(\bbH^2)^r\right)_{reg}$ and the \textit{singular limit set} is $\mathcal{L}_\Gamma^{sing}:=\mathcal{L}_\Gamma \backslash \mathcal{L}_\Gamma^{reg}$. 

We denote by $F_\Gamma$ the projection of $\mathcal{L}_\Gamma^{reg}$ on the Furstenberg boundary $\left( \partial \bbH^3 \right)^q\times \left( \partial \bbH^2 \right)^r$ and by $P_\Gamma$ the projection of $\mathcal{L}_\Gamma^{reg}$ on $\RP^{q+r-1}_+$. The projection $F_\Gamma$ is the \textit{Furstenberg limit set} of $\Gamma$ and $P_\Gamma$ is the \textit{projective limit set} of $\Gamma$.
\medskip

Let $h\in \Gamma$ be a loxodromic element or a mixed one with only hyperbolic or elliptic components. There is a unique unit speed geodesic $\gamma$ in $(\bbH^3)^q\times(\bbH^2)^r$ such that $h\circ\gamma(t) = \gamma(t + T_h)$ for a fixed $T_h \in \R_{>0}$ and all $t \in \R$. For $y \in \gamma$, the sequence $h^n(y)$ converges to $\gamma(+\infty)$. Hence also for every $x\in (\bbH^3)^q\times(\bbH^2)^r$, the sequence $h^n(x)$ converges to $\gamma(+\infty)$. Thus $\gamma(+\infty)$ is in $\mathcal{L}_\Gamma$. The sequence $h^{-n}(x)$ converges to $\gamma(-\infty):=-\gamma(+\infty)$ and therefore $\gamma(-\infty)$ is also in $\mathcal{L}_\Gamma$. The points $\gamma(+\infty)$ and $\gamma(-\infty)$ are the only fixed points of $h$ in $\Lim_\Gamma$. The point $\gamma(+\infty)$ is the \textit{attractive} fixed point of $h$ and the point $\gamma(-\infty)$ - the \textit{repulsive} fixed point of $h$. 

If $h$ is loxodromic, then for all $i=1,\ldots,q+r$, the projection $p_i \circ \gamma$ is not a point. Hence $\gamma$ is regular and $\gamma(+\infty) \in \mathcal{L}_\Gamma^{reg}$. The point $\gamma(+\infty)$ can be written as $(\xi_F,\xi_P)$ in our description of the regular geometric boundary where 
$$\xi_F := (p_1\circ \gamma(+\infty),\ldots,p_{q+r}\circ \gamma(+\infty))$$ 
is in the Furstenberg boundary and 
$$\xi_P := (d_H(p_1\circ \gamma(0),p_1\circ \gamma(1)) : \ldots :d_H(p_{q+r}\circ \gamma(0), p_{q+r}\circ \gamma(1)))$$ 
is in the projective limit set. Here we note that $\xi_P$ is also equal to 
$$(d_H(p_1\circ \gamma(0),p_1\circ \gamma(T_h)) : \ldots :d_H(p_{q+r}\circ \gamma(0), p_{q+r}\circ \gamma(T_h))),$$ 
which is exactly the translation direction of $h$.

Thus the translation direction of each loxodromic isometry $h$ in $\Gamma$ determines a point in the projective limit set $P_\Gamma$. This point does not change after conjugation with $h$ or after taking a power $h^m$ of $h$, because in these cases the translation direction remains unchanged. 

\medskip

Recall that following Maclachlan and Reid \cite{cM03}, we call a subgroup $\Gamma$ of $\PSL(2,\C)$ \textit{elementary} if there exists a finite $\Gamma$-orbit in $\overline{\bbH^3}:=\bbH^3 \cup \partial \bbH^3$ and \textit{nonelementary} if it is not elementary. Since $\bbH^3$ and $\partial \bbH^3$ are $\Gamma$-invariant, any $\Gamma$-orbit of a point in $\overline{\bbH^3}$ is either completely in $\bbH^3$ or completely in $\partial \bbH^3$. 

We call a subgroup $\Gamma$ of $\PSL(2,\C)^q\times\PSL(2,\R)^r$ \textit{nonelementary} if for all $i = 1, \ldots, q+r$,  $p_i(\Gamma)$ is nonelementary, and if for all $g \in \Gamma$ that are mixed, the projections $p_i\circ g$ are either loxodromic or elliptic of infinite order. Since for all $i = 1, \ldots, q+r$,  $p_i(\Gamma)$ is nonelementary, $\Gamma$ does not contain only elliptic isometries and thus $\Lim_\Gamma$ is not empty.

This definition of nonelementary is more restrictive than the one given by Link in \cite{gL02}. The definition of a nonelementary subgroup $\Gamma$ of $\PSL(2,\R)^r$ in \cite{gL02} is the following: The limit set of $\Gamma$ is nonempty and if $\xi \in \Lim_\Gamma$ and $\Gamma(\xi)$ denotes its  orbit under $\Gamma$, then each point in the orbit of $\xi$ under $\PSL(2,\C)^q\times\PSL(2,\R)^r$ can be \textit{connected} with a geodesic to at least one point in  $\Gamma(\xi)$. 

Two points $\xi$ and $\eta$ in $\partial((\bbH^3)^q\times(\bbH^2)^r)$ can be \textit{connected} by a geodesic if and only if $\xi=\gamma(\infty)$ and $\eta =\gamma(-\infty)$ for some geodesic $\gamma$ in $(\bbH^3)^q\times(\bbH^2)^r$. If $\xi$ and $\eta$ can be connected by a geodesic then they necessarily lie in the same orbit under $\PSL(2,\C)^q\times\PSL(2,\R)^r$. A possible element mapping $\xi$ to $\eta$ is one that fixes a point on the connecting geodesic and rotates around a geodesic orthogonal to $\gamma$ by $\pi$ in each of the first $q$ factors and around the fixed point by $\pi$ in the other $r$ factors.

\begin{lemma}
If a subgroup $\Gamma$ of $\PSL(2,\C)^q\times\PSL(2,\R)^r$ is nonelementary (according to our definition), then it is nonelementary in the sense of Link's definition in \cite{gL02}.
\end{lemma}
\begin{proof}
Let $\xi$ and $\eta$ be in the same orbit under $\PSL(2,\C)^q\times\PSL(2,\R)^r$ and $\gamma$ and $\delta$ their representative geodesic rays. Then $\xi$ and $\eta$ can be connected by a geodesic if and only if when $p_i\circ \gamma$ and $p_i\circ \delta$ are nonconstant in $\bbH^2$ (or $\bbH^3$), then $p_i\circ \gamma$ and $p_i\circ \delta$ are not in the same equivalence class in $\partial \bbH^2$ (or $ \partial \bbH^3$).

Let $\Gamma$ be nonelementary (according to our definition) and let $\xi$ be in $\mathcal{L}_\Gamma$ and $\eta$ a point in the orbit of $\xi$ under $\PSL(2,\C)^q\times\PSL(2,\R)^r$. Without loss of generality we can assume that the first $k$ projections of the defining geodesics of $\xi$ and $\eta$ are nonconstant. We denote their equivalence classes in $\partial((\bbH^3)^q\times(\bbH^2)^r)$ with $\xi_i$ and $\eta_i$ respectively, $i = 1,\ldots,k$. We will show by induction that there is an element $g =(g_1,\ldots,g_{q+r})$ of $\Gamma$ such that $g_i(\xi_i) \neq \eta_i$ for all $i=1,\ldots,k$ and therefore $\Gamma$ is nonelementary according to Link's definition.

First, for $j=1$, if $\xi_1 = \eta_1$, we can find $g \in \Gamma$ such that $g_1(\xi_1) \neq \eta_1$. The existence of $g$ follows from the fact that $p_1(\Gamma)$ is nonelementary and thus the orbit of $\xi_1$ under $p_1(\Gamma)$ is infinite.

Let $g \in \Gamma$ be such that $g_i(\xi_i) \neq \eta_i$ for all $i=1,\ldots,j$, $j<k$. If $g_{j+1}(\xi_{j+1}) \neq \eta_{j+1}$, then $g$ is the searched element. Otherwise, since $p_{j+1}(\Gamma)$ is nonelementary, there is $h =(h_1,\ldots,h_{q+r})$ in $\Gamma$ such that $h_{j+1}$ is loxodromic and does not have $\xi_{j+1}$ as a fixed point. Hence for all $n \in \N$, $h_{j+1}^n(\xi_{j+1}) \neq \eta_{j+1}$. 

According to our definition of nonelementary, for $i = 1,\ldots, j$, $h_i$ is either loxodromic or elliptic of infinite order.  Hence the point $g_i(\xi_i)$ is either a fixed point for $h_i$ or has an infinite orbit under $h_i$. In the first case, for any $n \in \N$, $h_i^n \circ g_i(\xi_i)\neq \eta_i$, and in the second case for $n$ big enough the same is true. Hence $h^n\circ g$ for $n$ big enough is the searched element.
\end{proof}

\begin{rem}
In the proof we used the assumption that for all $g \in \Gamma$ that are mixed, the projections $p_i\circ g$ are either loxodromic or elliptic of infinite order. We can prove the lemma without this assumption on $\Gamma$, but then the proof is a little more complicated because we need to consider different cases. And as we will see later we are only interested in groups $\Gamma$ such that  for all $g \in \Gamma$ that are mixed, the projections $p_i\circ g$ are either hyperbolic or elliptic of infinite order. 
\end{rem}

The next theorem is a special case of Theorem 3 from the introduction of \cite{gL02}. It describes the structure of the regular limit set of nonelementary discrete subgroups of $\PSL(2,\C)^q\times\PSL(2,\R)^r$.

\begin{theorem}[\cite{gL02}]
\label{T:LinkFP}
Let $\Gamma$ be a nonelementary discrete subgroup of the group $\PSL(2,\C)^q\times\PSL(2,\R)^r$ acting on $(\bbH^3)^q\times(\bbH^2)^r$. If $\mathcal{L}_\Gamma^{reg}$ is not empty, then $F_\Gamma$ is a minimal closed $\Gamma$-invariant subset of $(\partial \bbH^3)^q\times (\partial\bbH^2)^r$, the regular limit set equals the product $F_\Gamma \times P_\Gamma$ and $P_\Gamma$ is equal to the closure in $\RP^{q+r-1}_+$ of the set of translation directions of the loxodromic isometries in $\Gamma$. 
\end{theorem}

\section{Nonelementary groups}
\setcounter{theorem}{0}
In this part we show first that the regular limit set of a nonelementary subgroup of $\PSL(2,\C)^q\times\PSL(2,\R)^r$ is not empty and then we prove that its projective limit set is convex. This is a generalization of a result of Benoist in \cite{yB97}. Additionally, we describe the groups in which the projection to one factor is nonelementary.

\subsection{The regular limit set in nonempty}
To prove that the regular limit set of a nonelementary group is nonempty is equivalent in our case to proving that the group is \textit{strongly nonelementary}, i.e. that it contains a Schottky subgroup. In order to prove this we first need the next lemma.

Recall the definition of translation direction $L(g):= (\ell(g_1): \ldots: \ell(g_{q+r}))$.

\begin{lemma}
\label{L:Schottky}
Let $\Gamma$ be a nonelementary subgroup of $\PSL(2,\C)^q\times\PSL(2,\R)^r$. Further let $g$ and $h$ be two loxodromic isometries in $\Gamma$. Then there are loxodromic isometries $g'$ and $h'$ in $\Gamma$ with $L(g) = L(g')$ and $L(h) = L(h')$ such that the groups generated by the corresponding components are all Schottky groups (with only loxodromic isometries).
\end{lemma}
\begin{proof}
Let $g = (g_1,\ldots,g_{q+r})$ and $h = (h_1,\ldots,h_{q+r})$ be the given loxodromic isometries.

\textbf{Step 1.} We can assume that $g_1$ and $h_1$ do not have a common fixed point: Since $p_1(\Gamma)$ is nonelementary, there exists a transformation $\tilde{g} = (\tilde{g}_1,\ldots,\tilde{g}_{q+r})$ in $\Gamma$ such that $\tilde{g}_1$ is loxodromic and $\tilde{g}_1$ and $g_1$ do not have any common fixed point. Hence for $n\in \N$ big enough, the isometries $g_1$ and $\tilde{g}_1^nh_1\tilde{g}_1^{-n}$ do not have any common fixed point. Since the translation direction does not change under conjugation, we can consider $\tilde{g}^nh\tilde{g}^{-n}$ instead of $h$. 

\textbf{Step 2.} We can assume that $g_1$ and $h_1$ generate a Schottky group which contains only loxodromic isometries: Indeed, by the previous step, the isometries $g_1$ and $h_1$ do not have a common fixed point, therefore, by Lemma~\ref{L:SchottkyLox}, for $n$ big enough, $g_1^n$ and $h_1^n$ generate a Schottky group which contains only loxodromic isometries. Since $L(g) = L(g^n)$ and $L(h)= L(h^n)$, we take $g^n$ and $h^n$ instead of $g$ and $h$. 

\textbf{Step 3.} If $g$ and $h$ are as in the first and in the second step, then, for $i = 2,\ldots,q+r$, $g_i$ and $h_i$ have no common fixed point. In order to show this we assume that $g_i$ and $h_i$ have a common fixed point. Possibly after conjugation we can assume that this common fixed point is infinity and hence $g_i$ and $h_i$ are represented by the matrices $\begin{bmatrix} a & b \\ 0 & 1/a \end{bmatrix}$ and $\begin{bmatrix} c & d \\ 0 & 1/c \end{bmatrix}$ for some $a,c \in \R_{>0}\backslash \{1\}$ and $b,d \in \R$. Then
$$ [g_i,h_i] = g_ih_ig_i^{-1}h_i^{-1} = \begin{bmatrix} 1 & -cd-abc^2+a^2cd+ab \\ 0 & 1 \end{bmatrix}. $$
Hence the commutator $[g_i,h_i]$ is either parabolic or the identity. On the other hand $[g_i,h_i]$ has to be loxodromic or elliptic of infinite order because $[g_1,h_1]$ is a loxodromic isometry in the free group generated by $g_1$ and $h_1$. This is a contradiction.


Now we take $g' := g^N$ and $h':=h^N$ for $N\in \N$ big enough so that we can assure that for all $i =2,\ldots,q+r$, the group generated by $g'_i$ and $h'_i$ is a Schottky group with only loxodromic isometries. 
\end{proof}

The next lemma is needed in the proof of Lemma~\ref{L:NonemptyAlg}.

\begin{lemma}
\label{L:EllipticLox}
Let $g\in\PSL(2,\C)^q\times\PSL(2,\R)^r$ be elliptic of infinite order and $h\in\PSL(2,\C)^q\times\PSL(2,\R)^r$ be loxodromic. There are positive integers $m$ and $n$ such that $g^mh^n$ and $h^ng^m$ are loxodromic.
\end{lemma}
\begin{proof}
Since $g$ is an elliptic isometry of infinite order, there is a sequence $m_k$ such that $g^{m_k}$ converges to $Id$ when $k\longrightarrow \infty$. Now if $h$ is loxodromic, then $g^{m_k}h$ converges to $h$ when $k\longrightarrow \infty$. Hence $\tr(g^{m_k}h)\stackrel{k\rightarrow\infty}\longrightarrow\tr(h)$ and there is $K\in \N$ such that for all $k>K$, the isometries $g^{m_k}h$ and $hg^{m_k}$ are loxodromic. If this $K$ is big enough, then for all $n > 0$ and for all $k>K$, the isometries $g^{m_k}h^n$ and $h^ng^{m_k}$ are loxodromic.
\end{proof}

\begin{lemma}
\label{L:NonemptyAlg}
Let $\Gamma$ be a nonelementary subgroup of $\PSL(2,\C)^q\times\PSL(2,\R)^r$. Then $\mathcal{L}_\Gamma^{reg}$ is not empty.
\end{lemma}
\begin{proof}
The idea is to find an element $h$ in $\Gamma$ such that for all $i = 1,...,q+r$, the transformation $h_i$ is loxodromic. Then the attractive and repulsive fixed points of $h$ define a point in the regular limit set.

We are going to use a diagonal argument. For each $i = 1,\ldots, q+r$ we choose $g_i=(g_{i1},\ldots,g_{i,q+r})\in \Gamma$ such that $g_{ii}$ is loxodromic. The isometries $g_{i}$ do not need to be different. Using $g_i$, we will gradually construct the searched isometry $h$. In each step of the construction $g_{ii}$ will stay loxodromic.

First we show that for all $i = 2,\ldots,q+r$, we can choose the isometry $g_i$ so that $g_{i1}$ is loxodromic.

If for some $i = 2,\ldots,q+r$, $g_{i1}$ is not loxodromic then it is elliptic of infinite order. If $g_{1i}$ is loxodromic, then instead of $g_i$ we take $g_1$. In the other case, when $g_{1i}$ is elliptic, instead of $g_i$ we take $g_1^mg_i^n$ where $m$ and $n$ are chosen according to Lemma~\ref{L:EllipticLox} such that $g_{11}g_{i1}^n$ and $g_{1i}^mg_{ii}$ are loxodromic. Then also $g_{11}^mg_{i1}^n$ and $g_{1i}^mg_{ii}^n$ are loxodromic.

Now assume that for $i = k,\ldots,q+r$ and for $j = 1,\ldots,k-1$, the isometries $g_{ij}$ and $g_{ii}$ are loxodromic. We will show that for all $i = k+1,\ldots,q+r$, we can choose $g_i$ so that, for $j = 1,\ldots,k$, the isometries $g_{ij}$ and $g_{ii}$ are loxodromic. 

If for some $i = k+1,\ldots,q+r$, $g_{ki}$ is loxodromic, then instead of $g_i$ we take $g_k$. In the other case $g_{ki}$ is elliptic of infinite order. First instead of $g_k$ and $g_i$ we consider $g'_k$ and $g'_i$ that we get from Lemma~\ref{L:Schottky} after projecting $\Gamma$ on the first $k-1$ factors. The types of the isometries remain unchanged under conjugation. Instead of $g'_k$ and $g'_i$, we will continue to write $g_k$ and $g_i$.

Then using Lemma~\ref{L:SchottkyLox}, we take powers of $g_k$ and $g_i$ so that, for all $j = 1,\ldots,k-1$, $g_{kj}$ and $g_{ij}$ generate a Schottky group.

Finally, instead of $g_i$ we take $g_k^mg_i^n$ where $m$ and $n$ are chosen according to Lemma~\ref{L:EllipticLox} such that $g_{kk}g_{ik}^n$ and $g_{ki}^mg_{kk}$ are loxodromic. Then also $g_{kk}^mg_{ik}^n$ and $g_{ki}^mg_{ii}^n$ are loxodromic.
\end{proof}

\begin{rem}
This lemma is also true if we omit the condition for the mixed elements in the definition of a nonelementary group but the proof is more complicated. 
\end{rem}

If the group $\Gamma$ is a subgroup of $\PSL(2,\R)^r$ we have an even stronger statement.

\begin{lemma}
\label{L:Nonempty}
Let $\Gamma$ be a subgroup of $\PSL(2,\R)^r$ such that all mixed isometries in $\Gamma$ have only elliptic and hyperbolic components and $p_j(S)$ is nonelementary for one $j\in\{1,\ldots,r\}$. Then $\mathcal{L}_{\Gamma}^{reg}$ is not empty.
\end{lemma}
\begin{proof}
First we show that the subgroup $\Gamma$ of $\PSL(2,\R)^r$ contains a hyperbolic element $h=(h_1,\ldots,h_r)$. A proof of this fact is given also by Ricker in \cite{sR02}, in the proof of Proposition~2. 

Without loss of generality we can assume that $p_1(\Gamma)$ is nonelementary. In this case $\Gamma$ contains two isometries $g$ and $g'$ such that $g_1$ and $g'_1$ are hyperbolic without common fixed points. We set $\tilde{g}:=g'g{g'}^{-1}$. Then $g_1$ and $\tilde{g}_1$ also do not have any common fixed point. Again without loss of generality we can assume that, for $i = 1,\ldots, k$, $g_i$ and $\tilde{g}_i$ are hyperbolic and, for $i = k+1,\ldots, r$, elliptic of infinite order. By Lemma~\ref{L:SchottkyLox} and by Step 3 from the proof of Lemma~\ref{L:Schottky}, for $n$ big enough, for all $i = 1,\ldots, k$, the isometries $g_i^n$ and $\tilde{g}_i^n$ generate a Schottky group. Hence the isometries $[g_i^n, \tilde{g}_i^n] = g_i^n \tilde{g}_i^n g_i^{-n} \tilde{g}_i^{-n}$ are hyperbolic.

For $i = k+1,\ldots, r$, the isometries $g_i^n$ and $\tilde{g}_i^n$ do not commute, because otherwise their commutator $[g_i^n, \tilde{g}_i^n]$ will be the identity, which cannot be a component of a mixed isometry. Therefore $g_i^n$ and $\tilde{g}_i^n$ have different fixed points and by Theorem 7.39.2 in the book of Beardon \cite{aB95}, the commutator $[g_i^n$, $\tilde{g}_i^n]$ is hyperbolic.

Thus we have proved that $h:= g^n\tilde{g}^ng^{-n}\tilde{g}^{-n}$ is a hyperbolic element in $\Gamma$.

The attractive and repulsive fixed points of $h$ are points in the regular limit set of $p_1(\Gamma)$.
\end{proof}

As the next lemma shows, a corollary of the previous lemma is that $p_j(\Gamma)$ is ``big"(=non\-elementary) if and only if $\Gamma$ is ``big".

\begin{lemma}
\label{L:Nonelementary}
Let $\Gamma$ be a subgroup of $\PSL(2,\R)^r$ such that all mixed isometries in $\Gamma$ have only elliptic and hyperbolic components. If $p_j(\Gamma)$ is nonelementary for one $j\in\{1,\ldots,r\}$, then $\Gamma$ is also nonelementary.
\end{lemma}
\begin{proof}
Again without loss of generality we can assume that $j=1$. By the previous lemma, $\Gamma$ contains a hyperbolic element $h=(h_1,\ldots,h_r)$. Since $p_1(\Gamma)$ is nonelementary, the group $\Gamma$ contains an element $g=(g_1,\ldots,g_r)$ such that $g_1$ is hyperbolic and does not have any fixed point in common with $h_1$. Then for all $i = 2,\ldots,r$, the isometry $g_i$ is either elliptic of infinite order or it is hyperbolic that does not have common fixed points with $h_i$ (see Step 3 from the proof of Lemma~\ref{L:Schottky}). In both cases some powers of $h_i$ and $g_ih_ig_i^{-1}$ generate a Schottky group (by Lemma~\ref{L:SchottkyLox}.)
\end{proof}

Unfortunately, the above statement is false for subgroups of $\PSL(2,\C)^q\times\PSL(2,\R)^r$ with $q\geq 1$ and $q+r\geq2$.

\subsection{Groups with a nonelementary projection}
In this section we continue the investigation of groups with a nonelementary projection in one factor that was started with Lemma~\ref{L:Nonelementary}.

\begin{proposition}
\label{P:NonelemElliptic}
Let $\Gamma$ be a subgroup of $\PSL(2,\C)^q\times\PSL(2,\R)^r$ such that all mixed isometries in $\Gamma$ have only elliptic and loxodromic components and $p_j(\Gamma)$ is nonelementary for some $j\in\{1,\ldots,q+r\}$. Then each projection $p_i(\Gamma)$ to a real factor, i.e. $i=q+1,\ldots,q+r$, is nonelementary, and to a complex factor, i.e. $i=1,\ldots,q$, is either nonelementary or consists only of elliptic isometries with a common fixed point.
\end{proposition}
\begin{proof}
For simplicity we will denote $p_i(\Gamma)$ by $\Gamma_i$ for all $i = 1,\ldots,q+r$.

First we show that $\Gamma_i$ is nonelementary for $i=q+1,\ldots,q+r$. If $q+1 \leq j \leq q+r$, then from Lemma~\ref{L:Nonelementary} it follows that all $\Gamma_i$ are nonelementary for $i=q+1,\ldots,q+r$. Otherwise, we see that the proofs of Lemma~\ref{L:Nonempty} and Lemma~\ref{L:Nonelementary} work for the projections $p_i$ with $i=q+1,\ldots,q+r$ independently of the fact that the given nonelementary projection is not among them.

Without loss of generality, we can assume that $\Gamma_{q+r}$ is nonelementary. We will prove that for $i = 1,\ldots,q$, the subgroup $\Gamma_i$ of $\PSL(2,\C)$ is either nonelementary or consists only of elliptic isometries and then by Theorem 4.3.7 in Beardon's book \cite{aB95}, they have a common fixed point. In order to see this, we assume that $\Gamma_i$ contains a parabolic or loxodromic element $g_i$. It is the $i$-th component of an element $g=(g_1\ldots,g_{q+r})$ of $\Gamma$. Since $\Gamma_{q+r}$ is nonelementary, it contains a two-generated Schottky group with only hyperbolic isometries. Let $h=(h_1\ldots,h_{q+r})$ and $\tilde{h}=(\tilde{h}_1\ldots,\tilde{h}_{q+r})$ be the two isometries in $\Gamma$ such that $h_{q+r}$ and $\tilde{h}_{q+r}$ generate it.

By Theorem 4.3.5 in \cite{aB95}, if $h_i$ and $\tilde{h}_i$ have a common fixed point then their commutator has trace $2$, i.e. it is either the identity or parabolic. Since no component of a mixed isometry has trace $2$, the commutator of $h$ and $\tilde{h}$ can be only the identity or a parabolic isometry. Therefore the commutator of $h_{q+r}$ and $\tilde{h}_{q+r}$ has trace $2$ which is impossible because they generate a free group with only hyperbolic isometries. Hence $h_i$ and $\tilde{h}_i$ do not have a common fixed point in $\partial\bbH^3$.

If both $h_i$ and $\tilde{h}_i$ are loxodromic, then by Lemma~\ref{L:SchottkyLox} we see that they generate a Schottky group and hence $\Gamma_i$ is nonelementary.

If only one of them, let us say $h_i$, is loxodromic and the other one $\tilde{h}_i$ is elliptic of infinite order, then there is a power $k$ for which $h_i$ and $\tilde{h}^k_ih_i\tilde{h}_i^{-k}$ do not have a common fixed point and thus some powers of them generate by Lemma~\ref{L:SchottkyLox} a Schottky group and hence $\Gamma_i$ is nonelementary.

If both $h_i$ and $\tilde{h}_i$ are elliptic of infinite order and $g_i$ is loxodromic and does not have common fixed points with at least one of $h_i$ and $\tilde{h}_i$, we proceed as in the previous case in order to show that $\Gamma_i$ is nonelementary. Otherwise $g_i$ has one common fixed point with $h_i$ and one with $\tilde{h}_i$. There is a power $k$ for which $h_i$ and $\tilde{h}^k_ig_i\tilde{h}_i^{-k}$ do not have any common fixed point and we are in the previous case.

If $g_i$ is parabolic, then after conjugation we can assume that it is $\begin{bmatrix}1 & 1 \\ 0 & 1\end{bmatrix}$. For $h_i$ we have the representation $\begin{bmatrix}a & b \\ c & d\end{bmatrix}$ with $a+d$ real with absolute value less than $2$ and for $\tilde{h}_i$ the representation $\begin{bmatrix}\tilde{a} & \tilde{b} \\ \tilde{c} & \tilde{d}\end{bmatrix}$ with $\tilde{a}+\tilde{d}$ real with absolute value less than $2$. Since $h_i$ and $\tilde{h}_i$ do not have a common fixed point in $\partial\bbH^3$, they cannot both fix the point $\infty$. Therefore without loss of generality we can assume that $c$ is different from zero. Then
$$ g_i^kh_i = \begin{bmatrix}1 & k \\ 0 & 1\end{bmatrix}\begin{bmatrix}a & b \\ c & d\end{bmatrix} = \begin{bmatrix}a + kc & b + kd\\ c & d\end{bmatrix} $$
is loxodromic for $k=1$ if $c$ is not real and hyperbolic for $k$ big enough if $c$ is real. Hence we found a loxodromic element in $\Gamma_i$ and as in the previous case $\Gamma_i$ is nonelementary.
\end{proof}

\subsection{The projective limit set is convex}
Recall that the projective limit set $P_\Gamma$ is the projection of $L^{reg}_\Gamma$ on $\RP^{q+r-1}_+$. In this section we show that $P_\Gamma$ is ``nice", i. e. convex.

\begin{lemma}
\label{L:Connected}
Let $\Gamma$ be a nonelementary subgroup of $\PSL(2,\C)^q\times\PSL(2,\R)^r$ with $q+r \geq 2$. Then $P_\Gamma$ is convex in the real projective space $\RP^{q+r-1}$ with its standard metric and in particular $P_\Gamma$ is path connected.
\end{lemma}
\begin{proof}
The regular limit set $\mathcal{L}_\Gamma^{reg}$ is not empty by Lemma~\ref{L:NonemptyAlg}.

Since $\mathcal{L}_\Gamma^{reg}$ is not empty, $P_\Gamma$ contains at least one point. We will show that if $(x_1:\ldots:x_{q+r})$ and $(y_1:\ldots:y_{q+r})$  are two different points in $P_\Gamma$ then the segment $(x_1 + \lambda y_1:\ldots:x_{q+r} + \lambda y_{q+r})$ with $\lambda>0$ is also contained in $P_\Gamma$.

First we consider the case where $(x_1:\ldots:x_{q+r}) = (\ell(g_1):\ldots:\ell(g_{q+r}))$ and $(y_1:\ldots:y_{q+r}) = (\ell(h_1):\ldots:\ell(h_{q+r}))$ for loxodromic transformations $g,h$ in $\Gamma$.

By Lemma \ref{L:Schottky} we can assume that for all $i =2,\ldots,q+r$, the group generated by $g_i$ and $h_i$ is a Schottky group with only loxodromic isometries. 

Next we proceed as Dal'Bo in the proof of Proposition 4.2 in \cite{fD99} using Lemma 4.1 from \cite{fD99}. The latter says that there exists $C>0$ such that for all $m,n\in \N$ and all $i = 1,\ldots,q+r$,
$$ \left|\ell(g_i^mh_i^n) - m \ell(g_i) - n\ell(h_i) \right| <C.$$
Hence
$$ (\ell(g_1^{km}h_1^{kn}):\ldots:\ell(g_{q+r}^{km}h_{q+r}^{kn})) \stackrel{k\rightarrow\infty}{\longrightarrow} (\ell(g_1) + \frac{n}{m}\ell(h_1):\ldots:\ell(g_{q+r}) + \frac{n}{m}\ell(h_{q+r})) $$
and so for each $\lambda>0$, the point $(\ell(g_1) + \lambda\ell(h_1):\ldots:\ell(g_{q+r}) + \lambda\ell(h_{q+r}))$ is in the closure of the translation directions of the loxodromic isometries of $\Gamma$ and hence in $P_\Gamma$.

In this way we found a path in $P_\Gamma$ between $((\ell(g_1):\ldots:\ell(g_{q+r}))$ and $ (\ell(h_1):\ldots:\ell(h_{q+r}))$.

Now we consider two arbitrary different points $x$ and $y$ in $P_\Gamma$. By Link's Theorem \ref{T:LinkFP}, the translation directions of the loxodromic isometries of $\Gamma$ are dense in $P_\Gamma$. Therefore there are sequences $\{x^i\}$ and $\{y^i\}$ of translation directions of loxodromic isometries in $\Gamma$ such that $x^i \stackrel{i\rightarrow \infty}\longrightarrow x$ and $y^i \stackrel{i\rightarrow \infty}\longrightarrow y$.

We consider the canonical projections of $x^i$, $y^i$, $x$ and $y$ in the hyperplane $\{(t_1,\ldots,t_{q+r})\in \R^r\mid t_1+\cdots+t_{q+r} = 1\}$ and denote them by $(x^i_1,\ldots,x^i_{q+r})$,  $(y^i_1,\ldots,y^i_{q+r})$,  $(x_1,\ldots,x_{q+r})$ and  $(y_1,\ldots,y_{q+r})$ respectively. Then for all $j = 1,\ldots,r$, $x^i_j \stackrel{i\rightarrow \infty}\longrightarrow x_j$ and $y^i_j \stackrel{i\rightarrow \infty}\longrightarrow y_j$. Therefore $x^i_j +\lambda y^i_j \stackrel{i\rightarrow \infty}\longrightarrow x_j +\lambda y_j$ for $\lambda>0$. Hence 
$$ (x^i_1 +\lambda y^i_1:\ldots:x^i_{q+r} +\lambda y^i_{q+r})\stackrel{i\rightarrow \infty}\longrightarrow (x_1 +\lambda y_1:\ldots:x_{q+r} + \lambda y_{q+r}),\quad \text{with}\quad \lambda>0.$$
As already shown above the points $(x^i_1 +\lambda y^i_1:\ldots:x^i_{q+r} +\lambda y^i_{q+r})$ are in $P_\Gamma$. Hence the points $(x_1 +\lambda y_1:\ldots:x_{q+r} + \lambda y_{q+r})$ with $\lambda>0$ are also in $P_\Gamma$, which is what we wanted to show.
\end{proof}
\begin{rem}
This lemma is also true if we omit the condition for the mixed elements in $\Gamma$.
\end{rem}

\subsection{The limit cone and a theorem of Benoist}
\label{S:LimitConeBenoist}
The limit cone of a nonelementary group $\Gamma \leq \PSL(2,\C)^q\times\PSL(2,\R)^r$ with $q+r \geq 2$ is closely related to the projective limit set of $\Gamma$. 

In this section we first define the limit cone as defined by Benoist in \cite{yB97} and then show that the limit cone of a nonelementary $\Gamma$ is the closure of $P_\Gamma$ in $\RP^{q+r-1}$ and hence convex. We give also a version of the theorem in Section~1.2 in \cite{yB97} which motivated the previous result and is used later in this article.

The limit cone of $\Gamma$ is defined via the complete multiplicative Jordan decomposition of the elements of $\Gamma$. The multiplicative Jordan decomposition can be found for example in the the book of Eberlein \cite{pE96}. 

Each element $g \in \SL(2,\C)^q\times\SL(2,\R)^r$ can be decomposed in a unique way into $g = e_gh_gu_g$ where $e_g$ is elliptic (all eigenvalues have modulus 1), $h_g$ is hyperbolic (all eigenvalues are real positive), $u_g$ is unipotent (in our case parabolic) and all three commute. The canonical projection of $\SL(2,\C)$ into $\PSL(2,\C)$ gives the Jordan decomposition for $g \in \PSL(2,\C)^q\times\PSL(2,\R)^r$.

For instance, if $g \in \PSL(2,\C)$ is loxodromic, i.e. one of its representatives is conjugate to $\begin{bmatrix} e^{x+i\phi} & 0 \\ 0 & e^{-x-i\phi} \end{bmatrix}$ with $x,\phi \in \R$ then $e_g$ is conjugate to $\begin{bmatrix} e^{i\phi} & 0 \\ 0 &  e^{-i\phi} \end{bmatrix}$ and $h_g$ is conjugate to $\begin{bmatrix}e^x & 0 \\ 0 & e^{-x} \end{bmatrix}$. We have $x=\ell(g)/2$, where $\ell(g)$ is the translation length of $g$.

With $\lambda(g)$ we denote the unique element in $\PSL(2,\C)^q\times\PSL(2,\R)^r$ that is conjugate to $h_g$ and has diagonal form $\left(\begin{bmatrix}e^{x_1} & 0 \\0 & e^{-x_1}\end{bmatrix}, \ldots, \begin{bmatrix}e^{x_{q+r}} & 0 \\0 & e^{-x_{q+r}}\end{bmatrix}\right)$ with $x_i\geq 0$, $i=1,\ldots,r$. Then the  \textit{limit cone} of $\Gamma$ is the smallest closed cone in the space of diagonal elements in $\mathfrak{sl}(2,\C)^q\times\mathfrak{sl}(2,\R)^r$ that contains $\log(\lambda(\Gamma))$. E.g. if 
$$\lambda(g) = \left(\begin{bmatrix}e^{x_1} & 0 \\0 & e^{-x_1}\end{bmatrix}, \ldots, \begin{bmatrix}e^{x_{q+r}} & 0 \\0 & e^{-x_{q+r}}\end{bmatrix}\right)$$ 
with $x_1,\ldots,x_{q+r} \in \R_{\geq0}$, then 
$$\log(\lambda(g)) = \left(\begin{bmatrix} x_1 & 0 \\0 & -x_1\end{bmatrix}, \ldots, \begin{bmatrix} x_{q+r} & 0 \\0 & -x_{q+r}\end{bmatrix}\right)$$ 
and the limit cone contains the line 
$$\left(\begin{bmatrix} t x_1 & 0 \\0 & - t x_1\end{bmatrix}, \ldots, \begin{bmatrix} t x_{q+r} & 0 \\0 & - t x_{q+r}\end{bmatrix}\right),\quad t \in \R.$$

The translation direction of the isometry $g$ given above is $L(g)=(2x_1:\ldots:2x_{q+r})$. Hence the closure in $\RP^{q+r-1}$ of the translation directions of the hyperbolic and mixed elements of $\Gamma$ can be identified canonically with the limit cone of $\Gamma$.

The interior of the limit cone of $\Gamma$ is the intersection of the limit cone of $\Gamma$ with $\RP^{q+r-1}_+$. Hence it is exactly the projective limit set $P_\Gamma$ of $\Gamma$. In \cite{yB97} Benoist shows that for Zariski dense groups $\Gamma$, the limit cone of $\Gamma$ is convex and has nonempty interior. This means that the limit cone of $\Gamma$ is the closure of its interior. Thus for Zariski dense $\Gamma$, the limit cone of $\Gamma$ and $\overline{P_\Gamma}$ are the same. As the next theorem shows, the same is true even if $\Gamma$ is just nonelementary.

\begin{theorem}
\label{T:ConvexAlg}
Let $\Gamma$ be a nonelementary subgroup of $\PSL(2,\C)^q\times\PSL(2,\R)^r$ with $q + r \geq 2$. Then $P_\Gamma$ is convex and the closure of $P_\Gamma$ in $\RP^{q + r-1}$ is equal to the limit cone of $\Gamma$ and in particular the limit cone of $\Gamma$ is convex.
\end{theorem}
\begin{proof}
By Lemma~\ref{L:NonemptyAlg}, $\Lim_\Gamma^{reg}$ is nonempty. Then by Theorem 4.10 in \cite{gL06} the set of attractive fixed points of loxodromic isometries in a nonelementary subgroup of $\PSL(2,\C)^q\times\PSL(2,\R)^r$ is dense in its limit set. Hence the translation direction of every mixed isometry is the limit point of a sequence of translation directions of loxodromic isometries in $\Gamma$. From this it follows that the limit cone of $\Gamma$ is the closure of $P_\Gamma$ in $\RP^{q+r-1}$ and since by Lemma~\ref{L:Connected}, the projective limit set $P_\Gamma$ is convex, its closure in the convex set  $\RP^{q+r-1}$ is also convex.
\end{proof}
\begin{rem}
An alternative proof, which does not use Theorem 4.10 from \cite{gL06}, is given in \cite{sG09}, Theorem 3.8.
\end{rem}

This theorem extends partially the following special case of Benoist's theorem in Section~1.2 in \cite{yB97}.

\begin{theorem}[\cite{yB97}]
\label{T:BenoistLimitCone}
If $\Gamma$ is a Zariski dense over $\R$ subgroup of $\PSL(2,\C)^q\times\PSL(2,\R)^r$, $q + r \geq 2$, then the limit cone of $\Gamma$ is convex and its interior is not empty. 
\end{theorem}
\begin{rem}
Since the limit cone of $\Gamma$ is identified with $\overline{P_\Gamma}$, the interior of $P_\Gamma$ is also not empty. This is not always true for nonelementary groups $\Gamma$ that are not Zariski dense.
\end{rem}

\section{Irreducible arithmetic groups}
\setcounter{theorem}{0}
A special case of a general result of Margulis is that the irreducible lattices in $\PSL(2,\C)^q\times\PSL(2,\R)^r$ are all arithmetic. We present a construction of the irreducible arithmetic subgroups of $\PSL(2,\C)^q\times\PSL(2,\R)^r$ with the help of quaternion algebras. This construction gives a natural connection between some subgroups of $\PSL(2,\R)$ or $\PSL(2,\C)$ and irreducible arithmetic subgroups of $\PSL(2,\C)^q\times\PSL(2,\R)^r$.
\subsection{Quaternion algebras}
For more details concerning the definitions, notations and theorems in this section we refer to Katok's book \cite{sK92}, Chapter 5, and the book of Reid and Maclachlan \cite{cM03}, Chapters 0, 3 and 8. In this section $K$ will always denote a field.


A \textit{quaternion algebra over $K$} is a central simple algebra over $K$ which is four dimensional as a vector space over $K$.


Each quaternion algebra is isomorphic to an algebra $A=\left( \frac{a,b}{K} \right)$ with $a,b \in K^*=K\backslash\{0\}$ and a basis $\{1,i,j,k\}$, where $i^2 = a$, $j^2=b$, $k = ij = -ji$.

We denote by $M(2,K)$ the $2\times 2$ matrices with coefficients in the field $K$.

An isomorphism between $A = \left( \frac{a,b}{K} \right)$ and $M(2,K(\sqrt a))$ is given by the linear map sending the elements of the basis of $A$ to the following matrices:
$$ 1 \mapsto \begin{bmatrix} 1 & 0 \\ 0 & 1\end{bmatrix},\quad i \mapsto \begin{bmatrix} \sqrt a & 0 \\ 0 & -\sqrt a\end{bmatrix}, \quad j \mapsto \begin{bmatrix} 0 & 1 \\ b & 0\end{bmatrix},\quad k \mapsto \begin{bmatrix} 0 & \sqrt a \\ -b\sqrt a & 0\end{bmatrix}.$$
Thus if at least one of $a$ and $b$ is positive, then $\left( \frac{a,b}{\R} \right)$ is isomorphic to the matrix algebra $M(2,\R)$. If both $a$ and $b$ are negative, then $\left( \frac{a,b}{\R} \right)$ is isomorphic to the Hamilton quaternion algebra $\Ha = \left(\frac{-1,-1}{\R}\right)$.



$K$ is a \textit{totally real algebraic number field} if for each embedding of $K$ into $\C$ the image lies inside $\R$.

Let $A=\left( \frac{a,b}{K} \right)$ be a quaternion algebra. For every $x \in A$, $x = x_0+x_1i+x_2j+x_3k$, we define the \textit{reduced norm} of $x$ to be 
$$\Nrd(x) = x \bar{x} = x_0^2 - x_1^2 a - x_2^2 b + x_3^2 ab,$$ 
where $\bar{x}= x_0-x_1i-x_2j-x_3k$.


An \textit{order} $\mathcal{O}$ in a quaternion algebra $A$ over $K$ is a subring of $A$ containing $1$, which is a finitely generated $\mathcal{O}_K$-module and generates the algebra $A$ over $K$. (Here $\mathcal{O}_K$ denotes the ring of algebraic integers of $K$.)

The \textit{group of units} in $\mathcal{O}$ of reduced norm $1$ is $\mathcal{O}^1 = \{ \varepsilon \in \mathcal{O} \mid \Nrd(\varepsilon) = 1\}$.

We recall that a subgroup of $\SL(2,\C)$ is nonelementary if it does not have a finite orbit in its action on $\bbH^3\cup \partial \bbH^3$.

Let $\Gamma$ be a finitely generated nonelementary subgroup of $\SL(2,\C)$. Then the subgroup $\Gamma^{(2)}$ of $\Gamma$ generated by the set $\{g^2 \mid g \in \Gamma\}$ is a finite index normal subgroup of $\Gamma$.

Now we show how we can construct a quaternion algebra and an order starting from a finitely generated nonelementary subgroup of $\SL(2,\C)$. 

Let $\Gamma$ be a nonelementary subgroup of $\SL(2,\C)$. We denote
$$ A\Gamma :=  \{\sum a_i g_i \mid a_i \in \Q(\Tr(\Gamma)), g_i \in \Gamma\} $$
where only finitely many of the $a_i$ are non-zero. By Theorem 3.2.1 from \cite{cM03}, $A\Gamma$ is a quaternion algebra over $\Q(\Tr(\Gamma))$.

Two groups are \textit{commensurable} if their intersection has finite index in both of them. The \textit{commensurability class} of a subgroup $\Gamma$ of a group $G$ is the set of all subgroups of $G$ that are commensurable with $\Gamma$. The following theorem is Corollary 3.3.5 from \cite{cM03}. 
\begin{theorem}[\cite{cM03}]
\label{T:AlgebraInvariant}
Let $\Gamma$ be a finitely generated nonelementary subgroup of $\SL(2,\C)$. The quaternion algebra $A\Gamma^{(2)}$ is an invariant of the commensurability class of $\Gamma$.
\end{theorem}
The second theorem is Exercise 3.2, No. 1, in \cite{cM03} and a proof of it can be found in the proof of Theorem 8.3.2 in \cite{cM03}.
\begin{theorem}[\cite{cM03}]
\label{T:Order}
Let $\Gamma$ be a finitely generated nonelementary subgroup of $\SL(2,\C)$ such that all traces in $\Gamma$ are algebraic integers. Let also
$$ \mathcal{O}\Gamma :=  \{\sum a_i g_i \mid a_i \in \mathcal{O}_{\Q(\Tr(\Gamma))}, g_i \in \Gamma\} $$
where only finitely many of the $a_i$ are non-zero. Then $\mathcal{O}\Gamma$ is an order in $A\Gamma$.
\end{theorem}

\subsection{Irreducible arithmetic groups in $\PSL(2,\C)^q\times\PSL(2,\R)^r$}
In this section, following Schmutz and Wolfart \cite{pS00} and Borel \cite{aB81}, we will describe the irreducible arithmetic subgroups of $\PSL(2,\C)^q\times\PSL(2,\R)^r$.

Let $K$ be an algebraic number field of degree $n = [K:\Q]$ and let $\phi_i$, $i \in \{1, \ldots, n\}$, be the $n$ distinct embeddings of $K$ into $\C$, where $\phi_1 = id$. Further, we assume that $K$ has $q$ complex places. This means that $K$ has $2q$ different embeddings into $\C$ that can be divided into $q$ pairs of complex conjugated embeddings. 

If $K$ is not a subfield of $\R$, then for $i=1,\ldots,q$, let $\phi_i$ denote one of the embeddings in the pair and we assume that $\phi_1 = id$. Let $\phi_i$, $i = q+1,\ldots,n-q$, be the remaining embeddings of $K$ into $\C$ that are actually embeddings into $\R$. 

Let $A = \left(\frac{a,b}{K}\right)$ be a quaternion algebra over $K$ such that for $q+1 \leq i \leq q + r$, the quaternion algebra $\left(\frac{\phi_i(a),\phi_i(b)}{\R}\right)$ is \textit{unramified}, i.e. isomorphic to the matrix algebra $M(2, \R)$, and for $q + r < i \leq n-q$, it is \textit{ramified}, i.e. isomorphic to the Hamilton quaternion algebra $\Ha$. In other words, the embeddings 
\[\phi_i: K \longrightarrow \R, \quad i = q+1,\ldots, q+r\] 
extend to embeddings of $A$ into $M(2,\R)$ and the embeddings 
\[\phi_i: K \longrightarrow \R, \quad i = q+r+1,\ldots, n\] 
extend to embeddings of $A$ into $\Ha$. The embeddings 
\[\phi_i: K \longrightarrow \C, \quad i = 1,\ldots, q\]
extend to embeddings of $A$ into $M(2,\C)$. Note that the embeddings $\phi_i$, $i = 1,\ldots, q+r$, of $A$ into the matrix algebras $M(2,\C)$ and $M(2,\R)$ are not canonical. As we will see later, they are canonical up to conjugation and complex conjugation.

In the case when $K$ is a subfield of $\R$, the definition is analogous. We start with a number $r$, $1\leq r \leq n-2q$, and a quaternion algebra $A = \left(\frac{a,b}{K}\right)$ over $K$ such that for $1 \leq i \leq r$, the embeddings $\phi_i$ are embeddings of $K$ into $\R$ and the quaternion algebra $\left(\frac{\phi_i(a),\phi_i(b)}{\R}\right)$ is unramified. Each of the embeddings $\phi_i$, $i=r+1,\ldots,q+r$, is a representative of a different pair of complex conjugated embeddings. Finally, for $q + r < i \leq n-q$, the quaternion algebra $\left(\frac{\phi_i(a),\phi_i(b)}{\R}\right)$ is \textit{ramified}.

This case is interesting only when $K$ is a totally real algebraic number field because otherwise we can start with $\phi_i(K)$ that is not a subfield of the reals and consider the quaternion algebra $\left(\frac{\phi_i(a),\phi_i(b)}{\R}\right)$. So for simplicity of the notation we will assume that $\phi_1(K),\ldots,\phi_q(K)$ are not subfields of $\R$ and $\phi_{q+1}(K),\ldots,\phi_{q+r}(K)$ are subfields of $\R$ and $q$ or $r$ can be $0$.

Let $\mathcal{O}$ be an order in $A$ and $\mathcal{O}^1$ the group of units in $\mathcal{O}$. Define $\Gamma(A,\mathcal{O}):= \phi_1(\mathcal{O}^1) \subset \SL(2,\C)$. If $q=0$, then $\Gamma(A,\mathcal{O})$ is a subset of $\SL(2,\R)$. The canonical image of $\Gamma(A,\mathcal{O})$ in $\PSL(2,\C)$ is called a group \textit{derived from a quaternion algebra}. The group $\Gamma(A,\mathcal{O})$ acts by isometries on $(\bbH^3)^q\times (\bbH^2)^r$ as follows. An element $g = \phi_1(\varepsilon)$ of $\Gamma(A,\mathcal{O})$ acts via
\[g: (z_1, \ldots, z_{q+r}) \mapsto (\phi_1(\varepsilon)z_1, \ldots, \phi_{q+r}(\varepsilon)z_{q+r}), \]
where $z_i \mapsto \phi_i(\varepsilon)z_i$ is the usual action by linear fractional transformation, $i = 1, \ldots, {q+r}$.

For a subgroup $S$ of $\Gamma(A,\mathcal{O})$ we denote by $S^*$ the group
$$\{g^*:= (\phi_1(\varepsilon), \ldots, \phi_{q+r}(\varepsilon))\mid \phi_1(\varepsilon) = g \in S\}.$$
Instead of $(\phi_1(\varepsilon), \ldots, \phi_{q+r}(\varepsilon))$, we will usually write $(\phi_1(g), \ldots, \phi_{q+r}(g))$ or, since $\phi_1$ is the identity, even $(g,\phi_2(g), \ldots, \phi_{q+r}(g))$. The isometries $\phi_1(g), \ldots, \phi_{q+r}(g)$ are called \textit{$\phi$-conjugates}. 

Note that $g^*$ and $S^*$ depend on the chosen embeddings $\phi_i$ of $A$ into $M(2,\C)$ and $M(2,\R)$. On the other hand, the type of $g^*$ is determined uniquely by the type of $g$. This is given by the following lemma.

\begin{lemma}
Let $S$ be a subgroup of $\Gamma(A,\mathcal{O})$ and $S^*$ be defined as above. For an element $g\in S$ the following assertions are true.

1. If $g$ is the identity, then $g^*$ is the identity.

2. If $g$ is parabolic, then $g^*$ is parabolic.

3. If $g$ is elliptic of finite order, then $g^*$ is elliptic of the same order.

4. If $g$ is loxodromic, then $g^*$ is either loxodromic or mixed such that, for $i = 1,\ldots,q+r$, $\phi_i(g)$ is either loxodromic or elliptic of infinite order.

5. If $g$ is elliptic of infinite order, then its $\phi$-conjugates are loxodromic or elliptic of infinite order.
\end{lemma}
\begin{proof}
If $g$ is the identity (or elliptic of finite order), then $g^*$ is the identity (or elliptic of the same order) too, because each $\phi_i$ is an isomorphism between $A$ and $\phi_i(A)$.

If $g$ is parabolic, i.e. $\tr(g)= 2$, then $g^*$ is parabolic, because $\phi_i\left|_\Q\right. = id$ and $\phi_i$ has a trivial kernel.


If $g$ is loxodromic, then $g^*$ is either loxodromic or mixed such that, for $i = 1,\ldots,r$, $\phi_i(g)$ is either loxodromic or elliptic of infinite order. This is a consequence of the fact that the isometries of $\bbH^3$ (and $\bbH^2$) can be only loxodromic, elliptic of infinite order, elliptic of finite order, parabolic or the identity and the last three types of isometries are preserved under $\phi$-conjugation. Two examples show the two remaining possibilities. First, if $\tr(g) = 3+\sqrt 5 > 2$, i.e. $g$ is hyperbolic, then its $\phi$-conjugate with trace $3-\sqrt 5 < 2$ is elliptic of infinite order. The second example is with $\tr(g) = 6 + \sqrt 5$ and its $\phi$-conjugate with trace $6 - \sqrt 5$ are both hyperbolic. Note that if $g$ is hyperbolic, it could still have a $\phi$-conjugate that is purely loxodromic.

If $g$ is elliptic of infinite order, then its $\phi$-conjugates are loxodromic or elliptic of infinite order but at least one of them is loxodromic, because otherwise $S^*$ will not be discrete.
\end{proof}

Hence the mixed isometries in this setting have components that are only loxodromic or elliptic of infinite order. This justifies the condition in our definition of nonelementary that the projections of all mixed isometries can be only loxodromic or elliptic of infinite order.

By Borel \cite{aB81}, Section 3.3, all \textit{irreducible arithmetic subgroups} of the group $\PSL(2,\C)^q\times\PSL(2,\R)^r$ are commensurable to a $\Gamma(A,\mathcal{O})^*$. They have finite covolume. By Margulis, for $q+r\geq2$, all irreducible discrete subgroups of $\PSL(2,\C)^q\times\PSL(2,\R)^r$ of finite covolume are arithmetic, which shows the importance of the above construction.

We will mainly consider subgroups of irreducible arithmetic groups.

\medskip

P. Schmutz and J. Wolfart define in \cite{pS00} arithmetic groups acting on $(\bbH^2)^r$. Here we extend the definition for $(\bbH^3)^q\times(\bbH^2)^r$. An \textit{arithmetic group acting on} $(\bbH^3)^q\times(\bbH^2)^r$ is a group $G$ that is commensurable to a $\Gamma(A,\mathcal{O})$. It is finitely generated because it is commensurable to the finitely generated group $\Gamma(A,\mathcal{O})$. Then by Theorem \ref{T:AlgebraInvariant}, the quaternion algebra $AG^{(2)}$ of the group $G^{(2)}$ generated by the set $\{g^2 \mid g \in G\}$ is an invariant of the commensurability class of $G$. Hence $AG^{(2)}$ is isomorphic to $A$. By Theorem \ref{T:Order}, $\mathcal{O}G^{(2)}$ is an order in $AG^{(2)}$. Therefore $G^{(2)}$ is a subgroup of $\Gamma(AG^{(2)},\mathcal{O}G^{(2)})$ and we can define $(G^{(2)})^*$. This explains ``acting on $(\bbH^3)^q\times(\bbH^2)^r$" in the name.


The interest of this approach is that it allows us to consider a subgroup $S$ of $G\subset \PSL(2,\C)$ and then the corresponding subgroups ${S^{(2)}}^*$ and ${G^{(2)}}^*$ of $\PSL(2,\C)^q\times\PSL(2,\R)^r$. In this case $q$ and $r$ are determined by the bigger group $G$. 

\subsection{Arithmetic Fuchsian groups}
\label{S:ArithmGroups}
In the case of $r = 1$, the arithmetic subgroups of $\PSL(2,\R)$ are \textit{arithmetic Fuchsian group}. Since subgroups of arithmetic groups are sometimes also called arithmetic, we will emphasize the cofiniteness of an arithmetic Fuchsian group by calling it a cofinite arithmetic Fuchsian group. If $r=1$ then a subgroup of finite index of $\Gamma(A,\mathcal{O})$ and its canonical image in $\PSL(2,\R)$ are called \textit{Fuchsian groups derived from a quaternion algebra}.

The following characterization of cofinite arithmetic Fuchsian groups is due to Takeuchi \cite{kT75}.

\begin{theorem}[\cite{kT75}]
\label{T:TakeuchiArithmQuat}
Let $\Gamma$ be a cofinite Fuchsian group. Let $\Gamma^{(2)}$ be the subgroup of $\Gamma$ generated by the set $\{g^2 \mid g \in \Gamma\}$. Then $\Gamma$ is arithmetic if and only if the following two conditions are satisfied:

(i) $K:= \Q(\Tr(\Gamma^{(2)}))$ is an algebraic number field of finite degree and $\Tr(\Gamma^{(2)})$ is contained in the ring of integers $\mathcal{O}_K$ of $K$.

(ii) For any embedding $\varphi$ of $K$ into $\C$ which is not the identity, $\varphi(\Tr(\Gamma^{(2)}))$ is bounded in $\C$.
\end{theorem}

\begin{rem}
In \cite{kT75}, Takeuchi shows that $K$ is a totally real algebraic number field and the cofinite Fuchsian group $\Gamma^{(2)}$ satisfying (i) and (ii) is derived from a quaternion algebra over $K$. And since $\Gamma^{(2)}$ is of finite index in $\Gamma$, $\Gamma$ is arithmetic. Thus $\Gamma$ is arithmetic if and only if $\Gamma^{(2)}$ is derived from a quaternion algebra.
\end{rem}

The above characterization motivated the following definition of semi-arithmetic Fuchsian groups given in \cite{pS00}.

A cofinite Fuchsian group $\Gamma$ is \textit{semi-arithmetic} if and only if $K:= \Q(\Tr(\Gamma^{(2)}))$ is a totally real algebraic number field of finite degree $n = [ K : \Q ]$ and $\Tr(\Gamma^{(2)})$ is contained in the ring of integers $\mathcal{O}_K$ of $K$. $\Gamma$ is called \textit{strictly semi-arithmetic} if $\Gamma$ is not an arithmetic Fuchsian group.

The following theorem is a characterization of semi-arithmetic Fuchsian groups due to Schmutz Schaller and Wolfart \cite{pS00}.

\begin{theorem}[\cite{pS00}]
\label{T:SemiArithm}
Let $\Gamma$ be a cofinite Fuchsian group. Then the following two conditions are equivalent.

(i) $\Gamma$ is semi-arithmetic.

(ii) $\Gamma$ is commensurable to a subgroup $S$ of an arithmetic group $\Delta$ acting on $(\bbH^2)^r$.
\end{theorem}

\subsection{Arithmetic Kleinian groups}
In the case of $q = 1$ and $r=0$, we call a group commensurable to $\Gamma(A,\mathcal{O})$ an \textit{arithmetic Kleinian group}. If an arithmetic Kleinian group is additionally a subgroup of finite index of $\Gamma(A,\mathcal{O})$ then it is called \textit{Kleinian group derived from a quaternion algebra}.

The following characterization of cofinite arithmetic Kleinian groups is Theorem 8.3.2 in the book of Maclachlan and Reid \cite{cM03}.
\begin{theorem}[\cite{cM03}]
Let $\Gamma$ be a cofinite Kleinian group. Then $\Gamma$ is arithmetic if and only if the following three conditions hold.

(i) $\Q(\Tr(\Gamma^{(2)}))$ is an algebraic number field with exactly one complex place.

(ii) $\tr(g)$ is an algebraic integer for all $g \in \Gamma$.

(iii) $A\Gamma^{(2)}$ is ramified at all real places of $\Q(\Tr(\Gamma^{(2)}))$.
\end{theorem}

Corollary 8.3.5 in \cite{cM03} gives additionally that $\Gamma$ is arithmetic if and only if $\Gamma^{(2)}$ is derived from a quaternion algebra.

This characterization is very similar to Takeuchi's Theorem \ref{T:TakeuchiArithmQuat}. We can see it from the following formulation, which uses Lemma 5.1.3 from \cite{cM03}. 
\begin{theorem}[\cite{cM03}]
\label{T:MacReid}
Let $\Gamma$ be a cofinite Kleinian group. Then $\Gamma$ is arithmetic if and only if the following two conditions hold.

(i) $K=\Q(\Tr(\Gamma^{(2)}))$ is an algebraic number field and $\Tr(\Gamma^{(2)})$ is contained in the ring of integers $\mathcal{O}_K$ of $K$.

(ii) For any embedding $\varphi$ of $K$ into $\C$ which is neither the identity nor the complex conjugation, $\varphi(\Tr(\Gamma^{(2)}))$ is bounded in $\C$.
\end{theorem}

\section[Subgroups of arithmetic groups in $\PSL(2,\C)^q\times\PSL(2,\R)^r$]{Small limit sets of subgroups of arithmetic groups in $\PSL(2,\C)^q\times\PSL(2,\R)^r$}
\setcounter{theorem}{0}
A way to measure the size of a group is by the size of its limit set. It is still an open question how exactly to measure the size of the limit set. In our case we can say that a nonelementary group is small if its projective limit set is the smallest possible nonempty one, i.e. it is just a point, or if its Furstenberg boundary is measure 0 and even just a circle.

This section contains the main results of this article. We study the limit set of subgroups $\Gamma$ of arithmetic groups in $\PSL(2,\C)^q\times\PSL(2,\R)^r$ with $q+r\geq 2$ and determine for which groups the limit sets are the smallest.

First, we look at the projective limit set of a nonelementary $\Gamma$ and prove that it consists of exactly one point if and only if $p_j(\Gamma)$ is a subgroup of an arithmetic Fuchsian or Kleinian group for one $j\in\{1,\ldots,q+r\}$.

Then we show that the groups $\Gamma$ for which $p_j(\Gamma)$ is a subgroup of an arithmetic Fuchsian or Kleinian group are conjugate to an (almost) diagonal embedding of a Fuchsian or Kleinian arithmetic group and in particular that their limit set can be embedded as a topological space in a circle. This is not the case for the other groups.

\subsection{Examples: Triangle groups and Hilbert modular groups} 
A family of examples of nonelementary subgroups of arithmetic groups in $\PSL(2,\C)^q\times\PSL(2,\R)^r$ that are not Schottky groups is provided by the triangle Fuchsian groups. \textit{A Fuchsian triangle group of type} $(l,m,n)$ is a cofinite Fuchsian group generated by elliptic or parabolic elements $g$, $h$ and $s$ such that $ghs=id$, $g^l=id$, $h^m=id$ and $s^n=id$, where $1/l+1/m+1/n<1$.

For a more geometric definition we consider the group $S_0$ of reflections on the sides of a hyperbolic triangle with angles $\pi/l$, $\pi/m$ and $\pi/n$. Then the subgroup $S$ of $S_0$ of orientation preserving isometries is a Fuchsian triangle group of type $(l,m,n)$.

By Proposition 2 in Takeuchi \cite{kT77}, the ring $\Z[2\cos(\pi/l),2\cos(\pi/m),2\cos(\pi/n)]$ contains the trace set $\Tr(S)$, where $\pi/\infty=0$. In particular, the field $\Q(\Tr(S))$ coincides with the totally real algebraic number field $\Q(\cos(\pi/l),\cos(\pi/m),\cos(\pi/n))$ and $\Tr(S)$ is contained in the ring of integers of $\Q(\Tr(S))$. Hence $S$ is a semi-arithmetic group. By Takeuchi \cite{kT77}, only finitely many conjugacy classes of arithmetic triangle groups. 

By Theorem~\ref{T:SemiArithm}, $S$ is commensurable to a subgroup of an arithmetic group in $\PSL(2,\R)^r$.

\medskip
Examples of arithmetic subgroups $\Delta$ of $\PSL(2,\R)^r$ are the Hilbert modular groups. Let $F$ be a totally real number field and $\phi_i$, $i=1,\ldots,r$, be the $r$ distinct embeddings of $F$ into $\R$. For $g\in\PSL(2,\mathcal{O}_F)$, $g=\begin{bmatrix}a & b\\c&d\end{bmatrix}$, we define $\phi_i(g) = \begin{bmatrix}\phi_i(a) & \phi_i(b)\\\phi_i(c)&\phi_i(d)\end{bmatrix}$. The group $\PSL(2,\mathcal{O}_F)^*$ is an irreducible arithmetic subgroup of $\PSL(2,\R)^r$ and is called a \textit{Hilbert modular group} over $F$. Its quaternion algebra is isomorphic to $\left(\frac{1,1}{F}\right)$.

The Hilbert modular groups are the only arithmetic groups acting on $(\bbH^2)^r$ that contain parabolic isometries. Note that $\PSL(2,\Z)$ is a subgroup of any Hilbert modular group.

\medskip
A \textit{Hecke group} is a triangle group of type $(2,m,\infty)$. The Hecke groups are strictly semi-arithmetic except for $m=3,4,6$. 

A Hecke group $S$ of type $(2,m,\infty)$ is generated by $\begin{bmatrix}0 & 1\\-1&0\end{bmatrix}$ and $\begin{bmatrix}1 & 2\cos(\pi/m)\\0&1\end{bmatrix}$, see Katok~\cite{sK92}. Hence all elements in $S$ have entries that are algebraic integers in $\Q(\cos(\pi/m))$. Therefore $S$ is a subgroup of $\PSL(2,\mathcal{O}_F)$ where $F$ is a field which is a finite extension of $\Q(\cos(\pi/m))$. Hence $\PSL(2,\mathcal{O}_F)^*$ has a subgroup $\Gamma$ such that $p_1(\Gamma)=S$.

\subsection[Arithmetic Fuchsian/Kleinian groups in $\PSL(2,\C)^q\times\PSL(2,\R)^r$]{Arithmetic Fuchsian and Kleinian groups as subgroups of arithmetic groups in $\PSL(2,\C)^q\times\PSL(2,\R)^r$}
Let $\Gamma$ be a subgroup of an irreducible arithmetic group in $\PSL(2,\C)^q\times\PSL(2,\R)^r$. In this section we show that if $p_j(\Gamma)$ is a subgroup of an arithmetic Fuchsian or Kleinian group for some $j\in\{1,\ldots,q+r\}$ then the same is true for each nonelementary projection $p_i(\Gamma)$, $i=1,\ldots,q+r$.

\begin{lemma}
\label{L:ArithmAll}
Let $\Delta$ be an irreducible arithmetic subgroup of $\PSL(2,\C)^q\times\PSL(2,\R)^r$ and $\Gamma$ a subgroup of $\Delta$ such that $p_j(\Gamma)$ is a cofinite arithmetic Fuchsian (or Kleinian) group for some $j\in\{1,\ldots,q+r\}$. Then, for all $i=1,\ldots,q+r$, the group $p_i(\Gamma)$ is either elementary or a cofinite arithmetic Fuchsian (or Kleinian) group.
\end{lemma}
\begin{proof}
Since $\Delta$ is arithmetic, it is commensurable to an arithmetic group derived from a quaternion algebra $\Gamma(A,\mathcal{O})^*$. Hence for each $g = (g_1,\ldots,g_r)$ in $\Delta$ there is a power $k$ such that $g^k$ is in $\Gamma(A,\mathcal{O})^*$.

The group $\Gamma$ is commensurable to the subgroup $S^* = \Gamma \cap \Gamma(A,\mathcal{O})^*$ of $\Gamma(A,\mathcal{O})^*$. Then for all $i=1,\ldots,q+r$, the groups $p_i(\Gamma)$ and $\phi_i(S)$ are also commensurable. This in particular implies that $\phi_j(S)$ is a cofinite arithmetic Fuchsian (or Kleinian) group.

$\phi_j(S)^{(2)}$ is defined via a quaternion algebra $B$ over a field $k$ that is a subalgebra of $A$ and is ramified at all infinite places except one. This place is real if $\phi_j(S)$ is a cofinite Fuchsian and complex if $\phi_j(S)$ is a cofinite Kleinian group. The group $\phi_j(S)^{(2)}$ is isomorphic to the group of units of reduced norm 1 of an order $\mathcal{O}_B$ in $B$.

For $i=1,\ldots,q+r$, if $\phi_i(S)$ is nonelementary, then $B$ is unramified for the Galois' isomorphism $\tau_i:=\phi_i\circ\phi_j^{-1}$ and hence $\left.\tau_i\right|_k$ is the identity. Therefore $\phi_i(S)^{(2)}=\tau_i(\phi_j(S))^{(2)}$ is also isomorphic to the group of units of reduced norm 1 of $\mathcal{O}_B$ and hence $\phi_i(S)$ and $p_i(\Gamma)$ are cofinite arithmetic Fuchsian (or Kleinian) groups.

It remains to show that $\phi_i(S)$ is nonelementary if and only if $p_i(\Gamma)$ is nonelementary: Assume that $p_i(\Gamma)$ is nonelementary. Let $g$ and $h$ be two hyperbolic isometries that generate a Schottky group in $p_i(\Gamma)$. The isometries $g^{k_1}$ and $h^{k_2}$ are in $\phi_i(S)$ for some integers $k_1$ and $k_2$. Then $g^{k_1}$ and $h^{k_2}$ generate a Schottky subgroup of $\phi_i(S)$ and therefore $\phi_i(S)$ is nonelementary. The proof of the converse is analogous.
\end{proof}

\begin{rem}
From Proposition~\ref{P:NonelemElliptic} it follows that $p_{q+1}(\Gamma),\ldots,p_{q+r}(\Gamma)$ are nonelementary and hence of the same type as $p_j(\Gamma)$. But they can not be cofinite arithmetic Kleinian groups. Hence it is possible that $p_j(\Gamma)$ is a cofinite arithmetic Kleinian group only if $r=0$.
\end{rem}

The next lemma follows from the previous one.
\begin{lemma}
\label{L:ArithmAllAll}
Let $\Delta$ be an irreducible arithmetic subgroup of $\PSL(2,\C)^q\times\PSL(2,\R)^r$ and $\Gamma$ and $\widetilde{\Gamma}$ finitely generated subgroups of $\Delta$ such that $p_j(\Gamma)$ is a nonelementary subgroup of the cofinite arithmetic Fuchsian (or Kleinian) $p_j(\widetilde{\Gamma})$ for some $j\in\{1,\ldots,q+r\}$. Then, for all $i=1,\ldots,q+r$, the group $p_i(\Gamma)$ is either elementary or a nonelementary subgroup of a cofinite arithmetic Fuchsian (or Kleinian) group.
\end{lemma}
\begin{proof}
If $p_i(\Gamma)$ is nonelementary, then $p_i(\widetilde{\Gamma})$ is nonelementary and hence, by the previous lemma, a cofinite arithmetic Fuchsian (or Kleinian) group.

Let $p_i(\Gamma)$ be elementary. By Proposition~\ref{P:NonelemElliptic}, it consists only of elliptic isometries with a common fixed point. Since $p_j(\Gamma)$ has a loxodromic element, $p_i(\Gamma)$ has an elliptic element of infinite order. Hence $p_i(\Gamma)$ is not discrete. Therefore $p_i(\widetilde{\Gamma})$ is also not discrete. Since it is either elementary or a discrete arithmetic Fuchsian (or Kleinian) group, it is elementary.
\end{proof}

\subsection{Small projective limit sets}
In this section we will determine the groups for which the projective limit set is the smallest possible nonempty one, namely when it is just one point.

We need the following criterion for Zariski density which is a special case of the criterion proved by Dal'Bo and Kim in \cite{fD00}. For part (iii) we use the fact that there are no continuous isomorphisms between $\PSL(2,\R)$ and $\PSL(2,\C)$.

\begin{theorem}[\cite{fD00}]
\label{T:DalBoZariski}
(i) Let $\varphi$ be a surjective homomorphism between two Zariski dense subgroups $\Gamma$ and $\Gamma'$ of $\PSL(2,\R)$. Then $\varphi$ can be extended to a continuous automorphism of $\PSL(2,\R)$ if and only if the group $\Gamma_\varphi:= \{(g,\varphi(g))\mid g \in \Gamma\}$ is not Zariski dense in $\PSL(2,\R)\times\PSL(2,\R)$.

(ii) Let $\varphi$ be a surjective homomorphism between two subgroups $\Gamma$ and $\Gamma'$ of $\PSL(2,\C)$ that are Zariski dense over $\R$. Then $\varphi$ can be extended to a continuous automorphism of $\PSL(2,\C)$ if and only if the group $\Gamma_\varphi:= \{(g,\varphi(g))\mid g \in \Gamma\}$ is not Zariski dense over $\R$ in $\PSL(2,\C)\times\PSL(2,\C)$.

(iii) Let $\Gamma$ be a subgroup of $\PSL(2,\C)$ that is Zariski dense over $\R$ and $\Gamma'$ a Zariski dense subgroup of $\PSL(2,\R)$. Further let $\varphi:\Gamma \rightarrow \Gamma'$ be a surjective homomorphism between them. The group $\Gamma_\varphi:= \{(g,\varphi(g))\mid g \in \Gamma\}$ is Zariski dense over $\R$ in $\PSL(2,\C)\times\PSL(2,\R)$.
\end{theorem}

A proof of the following theorem is given by Schreier and Van der Waerden in \cite{oS28}.

\begin{theorem}[\cite{oS28}]
\label{T:PSLAutomorphisms}
(i) All continuous automorphisms of $\PSL(2,\R)$ are given by a conjugation with an element of $\GL(2,\R)$.

(ii) All continuous automorphisms of $\PSL(2,\C)$ are given by a conjugation with an element of $\GL(2,\C)$ or by a complex conjugation followed by a conjugation with an element of $\GL(2,\C)$.
\end{theorem} 

\subsubsection{The general case}
In the first four lemmas we will prove Theorem~\ref{T:ArithmCharactAlgAlg} which is the essential step of the proof of the main results for nonelementary subgroups of $\PSL(2,\C)^q\times\PSL(2,\R)^r$ with $q+r \geq 2$. Then we consider separately the three different cases: $q=0$, $r=0$ and $qr\neq 0$.

Unless otherwise specified, $\Gamma(A,\mathcal{O})$ will denote a subgroup of $\PSL(2,\R)$ or $\PSL(2,\C)$ derived from a quaternion algebra such that 
$$\Gamma(A,\mathcal{O})^* \subseteq \PSL(2,\C)^q\times\PSL(2,\R)^r \text{ with } q+r \geq 2.$$ 
Here we fix for simplicity of notation the order of the complex and real factors.

Let $S$ be a subgroup of $\Gamma(A,\mathcal{O})$ such that $S^*$ is nonelementary. Then by Lemma~\ref{L:NonemptyAlg} the regular limit set $\mathcal{L}_{S^*}^{reg}$ is not empty and in particular we can define the Furstenberg limit set $F_{S^*}$ and the projective limit set $P_{S^*}$.

In the first two lemmas we prove that $P_{S^*}$ contains exactly one point if and only if the $\phi$-conjugated elements have ``almost'' equal traces. Then we prove that this is the case if and only if $S$ is a subgroup of an arithmetic Fuchsian or Kleinian group.

\begin{lemma}
\label{L:MorePoints}
If for at least one $i\in\{1,\ldots,q+r\}$, the mapping $\phi_i:\Tr(S^{(2)}) \rightarrow \phi_i(\Tr(S^{(2)}))$ is neither the identity nor the complex conjugation, then $P_{S^*}$ contains more than one point.
\end{lemma}
\begin{proof}
We have four cases for $\Tr(S)$ and $\phi_i(\Tr(S))$.

The first one is when $\Tr(S)$ and $\phi_i(\Tr(S))$ are both subsets of $\R$. Then by Corollary~\ref{C:Zariski}, the Zariski closures over $\R$ of $S$ and $\phi_i(S)$ are conjugates of $\PSL(2,\R)$. Hence, by Theorem~\ref{T:DalBoZariski}(i) (the criterion of Dal'Bo and Kim), $S_{\phi_i} := \{(s,\phi_i(s))\mid s \in S\}$ is Zariski dense over $\R$ in a conjugate of $\PSL(2,\R)\times\PSL(2,\R)$.

The second case is when $\Tr(S)$ is not a subset of $\R$ and $\phi_i(\Tr(S))$ is a subset of $\R$. Then by Corollary~\ref{C:Zariski}, the Zariski closure over $\R$ of $S$ is $\PSL(2,\C)$ and the Zariski closure over $\R$ of $\phi_i(S)$ is a conjugate of $\PSL(2,\R)$. By Theorem~\ref{T:DalBoZariski}(iii) $S_{\phi_i}$ is then Zariski dense over $\R$ in a conjugate of $\PSL(2,\C)\times\PSL(2,\R)$

The third case is when $\Tr(S)$ is a subset of $\R$ and $\phi_i(\Tr(S))$ not is a subset of $\R$. It is analogous to the second case.

The last case is when both $\Tr(S)$ and $\phi_i(\Tr(S))$ are not subsets of $\R$. Then by Corollary~\ref{C:Zariski}, the Zariski closures over $\R$ of $S$ and $\phi_i(S)$ are $\PSL(2,\C)$. Hence, by Theorem~\ref{T:DalBoZariski}(ii) (the criterion of Dal'Bo and Kim), $S_{\phi_i}$ is Zariski dense over $\R$ in $\PSL(2,\C)\times\PSL(2,\C)$.

In all cases, by Theorem~\ref{T:BenoistLimitCone} it follows that $\overline{P_{S_{\phi_i}}}$ has a nonempty interior in $\RP^1$, i.e. $P_{S_{\phi_i}}$ contains more than one point.

Let $h$ be a loxodromic transformation in $S$ all of whose $\phi$-conjugates are loxodromic and whose existence is guaranteed by Lemma \ref{L:NonemptyAlg}. Since  $P_{S_{\phi_i}}$ contains more than one point, we can take $g \in S$ such that the translation directions $(\ell(h):\ell(\phi_i(h)))$ and $(\ell(g):\ell(\phi_i(g)))$ are different.

\textbf{Case 1:} If all $\phi$-conjugates of $g$ are loxodromic isometries, then the translation directions $L(h)$ and $L(g)$ of $h$ and $g$ determine different points in $\RP^{q+r-1}_+$ and hence $P_{S^*}$ consists of more than one point.

\textbf{Case 2:} There is a $\phi$-conjugate of $g$ that is an elliptic isometry of infinite order. By Theorem~\ref{T:ConvexAlg}, $\overline{P_{S^*}}$ is convex and in particular path connected. Hence there is a path in $P_{S^*}$ between $L(h)$ and $L(g)$. Since $\RP^{q+r-1}_+$ is open, there is an open subset of the path in $\RP^{q+r-1}_+$. Therefore there is another point in $P_{S^*}$ except $L(h)$.
\end{proof}

The converse is also true as the following lemma shows.
\begin{lemma}
\label{L:1:ldots:1}
If for all $i\in\{1,\ldots,q+r\}$, the mapping $\phi_i:\Tr(S^{(2)}) \rightarrow \phi_i(\Tr(S^{(2)}))$ is either the identity or the complex conjugation, then $P_{S^*}$ consists only of the point $(1:\ldots:1)$.
\end{lemma}
\begin{proof}
We prove the negation of this implication. Assume $P_{S^*}$ contains at least one point different from $(1:\ldots:1)$. By Theorem~\ref{T:LinkFP}, the translation directions of the loxodromic isometries in $S^*$ are dense in $P_{S^*}$. Therefore $P_{S^*}$ contains a loxodromic transformation $h^*$ with $L(h^*)\neq (1:\ldots:1)$. There is $\phi_i$ such that $\ell(h)\neq \ell(\phi_i(h))$, where $\ell(g)$ denotes the length of the closed geodesic corresponding to $g$.

For all $g, \tilde{g}\in\PSL(2,\C)$, if $\ell(g)\neq \ell(\tilde{g})$, then $\tr(g)\neq \tr(\tilde{g})$ and $\tr(g)\neq \overline{\tr(\tilde{g})}$. Therefore, for the above $\phi_i$, we have $\tr(h)\neq \pm \phi_i(\tr(h))$ and $\tr(h)\neq \pm \overline{\phi_i(\tr(h))}$ and in particular the mapping $\phi_i:\Tr(S^{(2)}) \rightarrow \phi_i(\Tr(S^{(2)}))$ is neither the identity nor the complex conjugation.
\end{proof}

\begin{rem}
The last two lemmas are generally true if $S$ is nonelementary and $\phi_i$ for $i=1,\ldots,q+r$ are group isomorphisms such that $\phi_i(S)$ are nonelementary.
\end{rem}

We will need the following lemma in the proof of Lemma~\ref{L:NotIDNotArithm}.
\begin{lemma}
\label{L:IntegerTraces}
Let $\Delta$ be an irreducible arithmetic subgroup of $\PSL(2,\C)^q\times\PSL(2,\R)^r$. Then for all $g=(g_1,\ldots,g_{q+r})\in \Delta$, the traces $\tr(g_1),\ldots,\tr(g_{q+r})$ are algebraic integers.
\end{lemma}
\begin{proof}
The group $\Delta$ is commensurable to a $\Gamma(A,\mathcal{O})^*$. Then for all $g\in\Delta$, there is a power $g^n\in \Gamma(A,\mathcal{O})^*$ for some $n\in \N$.  

From Lemma~2.2.7 and Lemma~2.2.4 in \cite{cM03} it follows that the traces of all elements in $\Gamma(A,\mathcal{O})$ are algebraic integers. Hence $\tr(g_i^n)$, $i=1,\ldots,q+r$, are algebraic integers. Since $\tr(g_i^n)$ is a monic polynomial with integer coefficients in $\tr(g_i)$, the trace $\tr(g_i)$ satisfies a monic polynomial with coefficients that are algebraic integers and hence is an algebraic integer.
\end{proof}

\begin{lemma}
\label{L:NotIDNotArithm}
Let $S$ be finitely generated. Then the mapping $\phi_i:\Tr(S^{(2)}) \rightarrow \phi_i(\Tr(S^{(2)}))$ is neither the identity nor the complex conjugation for at least one $\phi_i$, $i\in\{1,\ldots,q+r\}$, if and only if $S$ is not contained in an arithmetic Fuchsian or Kleinian group.
\end{lemma}
\begin{proof}
In order to prove the first implication we assume that the mapping $\phi_i:\Tr(S^{(2)}) \rightarrow \phi_i(\Tr(S^{(2)}))$ is neither the identity nor the complex conjugation for at least one $\phi_i$, $i\in\{1,\ldots,q+r\}$. In this case $S$ is not contained in an arithmetic Kleinian or Fuchsian group because there is the embedding $\phi_i$ of $\Q(\Tr(S^{(2)}))$ into $\C$ which is neither the identity nor the complex conjugation such that $\phi_i(\Tr(S^{(2)}))$ is not bounded in $\C$ and this contradicts the second condition in Theorem~\ref{T:TakeuchiArithmQuat} and Theorem~\ref{T:MacReid}.

We prove the negation of the second implication. We assume that for all $i=1,\ldots,q+r$ we have $\phi_i\left|_{\Tr(S^{(2)})}\right.=id$, i.e. $\tr(g)=\pm\phi_i(\tr(g))$, or $\phi_i\left|_{\Tr(S^{(2)})}\right.$ is the complex conjugation, i.e. $\overline{\tr(g)}=\pm\phi_i(\tr(g))$ for all $g\in \Gamma$. 

We consider the following set of matrices
$$ AS^{(2)} = \{\sum a_i g_i \mid  a_i \in \Q(\Tr(S^{(2)})), g_i \in S^{(2)}\}$$
where only finitely many of the $a_i$ are nonzero. 
By Theorem~\ref{T:AlgebraInvariant}, $AS^{(2)}$ is a quaternion algebra over $\Q(\Tr(S^{(2)}))$ because $S$ is finitely generated. By construction it is contained in $A$. So it is a quaternion algebra over the algebraic number field $\Q(\Tr(S^{(2)}))$ which is unramified at $id$ and ramified at all other infinite places. By Lemma~\ref{L:IntegerTraces} all traces in $S$ are algebraic integers and hence by Theorem \ref{T:Order}, an order of $AS^{(2)}$ is
$$ \mathcal{O}S^{(2)} = \{\sum a_i g_i \mid  a_i \in \mathcal{O}_{\Q(\Tr(S^{(2)}))}, g_i \in S^{(2)}\}$$
where only finitely many of the $a_i$ are nonzero. The group $\mathcal{O}{S^{(2)}}^1:=\{\alpha \in \mathcal{O}S \mid \Nrd(\alpha) = 1\}$ is an arithmetic Kleinian or Fuchsian group depending on whether its trace field is a subset of $\R$ or not. The group $S^{(2)}$ is contained in $\mathcal{O}{S^{(2)}}^1$.

It remains to construct an arithmetic Fuchsian group containing $S$. Since $S$ is finitely generated, we can assume that it is generated by its elements $h_1,\ldots,h_m$. Note that $h_i^2\in \mathcal{O}{S^{(2)}}^1$ because $h_i^2\in S^{(2)}$ for all $i=1,\ldots,m$. We consider the group $\widetilde{S}$ generated by $\mathcal{O}{S^{(2)}}^1$ and $h_1,\ldots,h_m$. It is finitely generated and nonelementary. Hence $\widetilde{S}^{(2)}$ is a finite index subgroup of $\widetilde{S}$. On the other hand we have the group inclusions $\widetilde{S}^{(2)} \leq \mathcal{O}{S^{(2)}}^1 \leq \widetilde{S}$. Therefore $\mathcal{O}{S^{(2)}}^1$ is a finite index subgroup of $\widetilde{S}$ and $\widetilde{S}$ is an arithmetic Kleinian or Fuchsian subgroup of $\Gamma(A,\mathcal{O})$.
\end{proof}
\begin{rem}
We also proved the statement that $S$ is contained in an arithmetic Fuchsian or Kleinian group if and only if $S$ is contained in an arithmetic Fuchsian or Kleinian subgroup of $\Gamma(A,\mathcal{O})$. 
\end{rem}

The following theorem follows directly from Lemma~\ref{L:MorePoints}, Lemma~\ref{L:1:ldots:1} and Lemma~\ref{L:NotIDNotArithm}.

\begin{theorem}
\label{T:ArithmCharactAlgAlg}
Let $\Gamma(A,\mathcal{O})$ be a subgroup of $\PSL(2,\C)$ or $\PSL(2,\R)$ derived from a quaternion algebra such that $\Gamma(A,\mathcal{O})^*\subset\PSL(2,\C)^q\times \PSL(2,\R)^r$ with $r+q\geq 2$ and let $S$ be a finitely generated subgroup of $\Gamma(A,\mathcal{O})$ such that $S^*$ is nonelementary. Then $\mathcal{L}_{S^*}^{reg}$ is not empty and $P_{S^*}$ consists of exactly one point if and only if $S$ is contained in an arithmetic Fuchsian or Kleinian group. 
\end{theorem}
\begin{rem}
From Theorem~\ref{T:ArithmCharactAlgAlg} follows in particular that if $S$ is not a Fuchsian group, i.e. $S$ is not discrete, then  $P_{S^*}$ contains more than one point.
\end{rem}

\begin{theorem}
\label{T:MainCharactAlg}
Let $\Delta$ be an irreducible arithmetic subgroup of $\PSL(2,\C)^q\times \PSL(2,\R)^r$ with $q+r\geq2$ and $\Gamma$ a finitely generated nonelementary subgroup of $\Delta$. Then $\mathcal{L}_{\Gamma}^{reg}$ is not empty and $P_{\Gamma}$ consists of exactly one point if and only if $p_j(\Gamma)$ is contained in an arithmetic Fuchsian or Kleinian group for some $j\in\{1,\ldots,q+r\}$.
\end{theorem}
\begin{proof}
By Lemma~\ref{L:ArithmAll}, the group $p_j(\Gamma)$ is contained in an arithmetic Fuchsian or Kleinian group if and only if the group $p_1(\Gamma)$ is contained in an arithmetic Fuchsian or Kleinian group. Thus we prove the statement with $p_1(\Gamma)$ instead of $p_j(\Gamma)$.

We recall that $\mathcal{L}_{\Gamma}^{reg}$ is not empty by Lemma~\ref{L:Nonempty}.

Since $\Delta$ is arithmetic, it is commensurable with an arithmetic group derived from a quaternion algebra $\Gamma(A,\mathcal{O})^*$. Hence there is $k\in\N$ such that, for each $g = (g_1,\ldots,g_{q+r})$ in $\Gamma$, $g^k$ is in $\Gamma(A,\mathcal{O})^*$.

There is a subgroup $S$ of $\Gamma(A,\mathcal{O})$ such that $S^* = \Gamma \cap \Delta \cap \Gamma(A,\mathcal{O})^*$. The group $\Gamma$ is commensurable with the subgroup $S^*$. Then $p_1(\Gamma)$ and $S$ are also commensurable. The group $S^*$ is finitely generated because it is a finite index subgroup of the finitely generated group $\Gamma$. (This follows from the Schreier Index Formula, see for example the book \cite{jS93}, 2.2.5.) The group $S^*$ is also nonelementary because $\Gamma$ is nonelementary: Let $g$ and $h$ be two loxodromic isometries that generate a Schottky group in $\Gamma$. The isometries $g^{k}$ and $h^{k}$ are in $S^*$. Then $g^{k}$ and $h^{k}$ generate a Schottky subgroup of $S^*$.

Thus $S^*$ is a  finitely generated nonelementary subgroup of $\Gamma(A,\mathcal{O})$. By Theorem~\ref{T:ArithmCharactAlgAlg}, the group $S$ is contained in an arithmetic Fuchsian or Kleinian group if and only if its projective limit set $P_{S^*}$ contains exactly one point. 

The final step is to go back to $\Gamma$.

If $P_{\Gamma}$ contains at least two points, then it contains two points that are the translation directions of two loxodromic isometries $g$ and $h$ of $\Gamma$. The isometries $g^{k}$ and $h^{k}$ that are in $S^*$ have the same translation directions as $g$ and $h$. Hence $L(g^{k})$ and $L(h^{k})$ are different points in $P_{S^*}$ and therefore $S$ is not a subgroup of an arithmetic Fuchsian or Kleinian group. Thus $p_1(\Gamma)$ is not a subgroup of an arithmetic Fuchsian or Kleinian group too.

If $P_{\Gamma}$ contains exactly one point, then $P_{S^*}$ contains also exactly one point and $S$ is contained in an arithmetic Fuchsian or Kleinian group. Hence the quaternion algebra $B:=AS^{(2)}=Ap_1(\Gamma)^{(2)}$, which is an invariant of the commensurability class, is unramified only at one place. Since by Lemma~\ref{L:IntegerTraces} the trace set $\Tr(p_1(\Gamma)^{(2)})$ consists of algebraic integers, $\mathcal{O}p_1(\Gamma)^{(2)}$ is an order in $B$ and thus $p_1(\Gamma)^{(2)}$ is a subgroup of an arithmetic Fuchsian or Kleinian group. The proof that $p_1(\Gamma)$ is a subgroup of an arithmetic Fuchsian or Kleinian group is the same as the one for $S$ in the end of Lemma~\ref{L:NotIDNotArithm}.
\end{proof}

\subsubsection{Subgroups of $\PSL(2,\R)^r$}
In the case when $q=0$ we can specify the statement of Theorem~\ref{T:MainCharactAlg}. First, we have Lemma~\ref{L:Nonempty}, so requiring that $\Gamma$ is nonelementary is equivalent to requiring that one of its projections is nonelementary. And second, $\PSL(2,\R)$ does not have arithmetic Kleinian subgroups. Hence we have proved the following corollary.

\begin{corollary}
\label{C:MainCharactAlgR}
Let $\Delta$ be an irreducible arithmetic subgroup of $\PSL(2,R)^r$ with $r\geq2$ and $\Gamma$ a finitely generated nonelementary subgroup of $\Delta$. Then $\mathcal{L}_{\Gamma}^{reg}$ is not empty and $P_{\Gamma}$ consists of exactly one point if and only if $p_j(\Gamma)$ is contained in an arithmetic Fuchsian group for some $j\in\{1,\ldots,r\}$. 
\end{corollary}
\begin{rem}
This corollary is in particular true when $p_j(\Gamma)$ is a cofinite Fuchsian group.
\end{rem}

\subsubsection{Subgroups of $\PSL(2,\C)^q\times\PSL(2,\R)^r$}
In the case when both $q$ and $r$ are at least $1$, we can state Theorem~\ref{T:MainCharactAlg} more precisely because by the remark after Lemma~\ref{L:ArithmAll}, $p_j(\Gamma)$ can not be a cofinite arithmetic Kleinian group for any group $\Gamma$.
\begin{corollary}
\label{C:MainCharactAlgAlg}
Let $\Delta$ be an irreducible arithmetic subgroup of $\PSL(2,\C)^q\times\PSL(2,\R)^r$ with $q,r\geq1$ and $\Gamma$ a finitely generated nonelementary subgroup subgroup of $\Delta$. Then $\mathcal{L}_{\Gamma}^{reg}$ is not empty and $P_{\Gamma}$ consists of exactly one point if and only if $p_j(\Gamma)$ is contained in an arithmetic Fuchsian group for some $j\in\{1,\ldots,q+r\}$. 
\end{corollary}

\begin{rem}
It is not possible to prove an analogous statement to Lemma~\ref{L:Nonempty} for subgroups of $\PSL(2,\C)^q\times\PSL(2,\R)^r$ because if $S$ is nonelementary subgroup of a $\Gamma(A,\mathcal{O})$, then $S^*$ is not necessarily nonelementary.

An example is the quaternion algebra $\left(\frac{\sqrt 2,-1}{\Q(\sqrt 2)}\right)$ with its embedding $\phi_1$ in $M(2,\Q(\sqrt[4] 2))$ given by the linear map sending the elements of the basis of $A$ to the following matrices:
$$ 1 \mapsto \begin{bmatrix} 1 & 0 \\ 0 & 1\end{bmatrix},\quad i \mapsto \begin{bmatrix} \sqrt[4] 2 & 0 \\ 0 & -\sqrt[4] 2\end{bmatrix}, \quad j \mapsto \begin{bmatrix} 0 & 1 \\ -1 & 0\end{bmatrix},\quad k \mapsto \begin{bmatrix} 0 & \sqrt[4] 2 \\ \sqrt[4] 2 & 0\end{bmatrix}.$$

The finitely generated $\mathcal{O}_{\Q(\sqrt 2)}$-module $\mathcal{O}=\{x=x_0+x_1i+x_2j+x_3k\mid x_0,x_1,x_2,x_3\in\mathcal{O}_{\Q(\sqrt 2)}\}$ is a ring containing 1 and hence an order because $i^2$ and $j^2$ are algebraic integers. The group $S=\phi_1(\mathcal{O}^1)$ is an arithmetic Fuchsian group because its $\phi$-conjugate is a subgroup of $\left(\frac{-\sqrt 2,-1}{\Q(\sqrt 2)}\right)$, which is isomorphic to the Hamilton quaternion algebra $\Ha$.

The group $S$ is a subgroup of the arithmetic group acting on $(\bbH^2)^2\times\bbH^3$
 $$ \Delta =\{ A \in M(2,\Q(\sqrt[4] 2))\mid \det A = 1\}.$$
The group $S^*$ is not nonelementary because $\phi_3(S)$, which is a subgroup of $\PSL(2,\C)$, consists only of elliptic isometries.
 
We can construct some other examples by instead of taking $ M(2,\Q(\sqrt[4] 2))$ we take $M(2,K)$ where $K$ is a finite extension of $\Q(\sqrt[4] 2)$.
\end{rem}

\subsubsection{Subgroups of $\PSL(2,\C)^q$}
In the case $r=0$, Theorem~\ref{T:MainCharactAlg} is stated in the most general way. This case is of independent interest because this is the only case when $p_j(\Gamma)$ can be a cofinite arithmetic Kleinian group.

\subsection{Small limit cones}
The restriction in Theorem~\ref{T:MainCharactAlg} that $\Gamma$ should be nonelementary is needed so that $\Lim_\Gamma^{reg}$ is nonempty and hence $P_\Gamma$ is well defined. This can be avoided by using the limit cone of $\Gamma$ as defined in Section~\ref{S:LimitConeBenoist}. This is proved in Theorem~\ref{T:CharactLimitCone}. In order to do so, we first need two lemmas.

\begin{lemma}
\label{L:AtLeastTwoNonelem}
Let $\Gamma$ be a finitely generated subgroup of an irreducible arithmetic group $\Delta$ in $\PSL(2,\C)^q\times\PSL(2,\R)^r$ with $q+r\geq 2$. If $p_j(\Gamma)$ for some $j\in\{1,\ldots,q+r\}$ is nonelementary and not a subgroup of an arithmetic Fuchsian or Kleinian group, then there is $i\in\{1,\ldots,q+r\}$, $i\neq j$, such that $p_i(\Gamma)$ is nonelementary.
\end{lemma}
\begin{proof}
Since $\Delta$ is arithmetic, it is commensurable with an arithmetic group derived from a quaternion algebra $\Gamma(A,\mathcal{O})^*$. 

We assume that $p_i(\Gamma)$ is elementary for all $i$ except $j$. Then for the subgroup $S$ of $\Gamma(A,\mathcal{O})$ defined as $S^* = \Gamma \cap \Delta \cap \Gamma(A,\mathcal{O})^*$ only $\phi_j(S)$ is nonelementary. This means that for all embeddings $\sigma$ of the field $F:=\Q(\Tr(\phi_j(S)^{(2)}))$ into $\C$ that are not the identity or the complex conjugation, the set $\sigma(\Tr(\phi_j(S)^{(2)}))$ is bounded. By Lemma~\ref{L:IntegerTraces}, the trace set $\Tr(\phi_j(S)^{(2)})$ consists of algebraic integers. Since the properties nonelementary and finitely generated are invariant in the commensurability class (see the proof of Theorem~\ref{T:MainCharactAlg}), the group $S^{(2)}$ is also nonelementary and hence by Lemma~5.1.3 and Corollary~8.3.7 in \cite{cM03} it is a subgroup of an arithmetic Kleinian or Fuchsian group. Then $p_j(\Gamma)$ is a subgroup of an arithmetic Fuchsian or Kleinian group, which is a contradiction.
\end{proof}

The second lemma explains how to make nonelementary a subgroup $\Gamma$ of $\Delta$ that is not nonelementary. In order to do so, we need to consider $\Gamma$ as a subgroup of the product group $\PSL(2,\C)^{q'}\times\PSL(2,\R)^{r}$ with $q'<q$.

We remark that since all mixed isometries in $\Gamma$ have only components that are loxodromic and elliptic of infinite order, we have $\Gamma = \left\{(\phi_1(g_1), \phi_2(g_1), \ldots, \phi_{q+r}(g_1)) \mid g_1 \in p_1(\Gamma)\right\}$ where $\phi_i$ is a surjective homomorphism between $p_1(\Gamma)$ and $p_i(\Gamma)$. Here $\phi_i$ coincides with the $\phi_i$ coming from $\Gamma(A,\mathcal{O})^*$ for $\Gamma \cap \Gamma(A,\mathcal{O})^*$.

If for at least one $j\in\{1,\ldots,q+r\}$, the projection $p_j(\Gamma)$ is nonelementary, define 
$$\Gamma^{ne}:= \left\{(\phi_{i_1}(g_1), \phi_{i_2}(g_1), \ldots, \phi_{i_n}(g_1)) \mid~i_1<\ldots<i_n \text{ and } p_{i_k}(\Gamma) \text{ nonelementary}\right\}.$$

\begin{lemma}
\label{L:GammaNE}
The group $\Gamma^{ne}$ is nonelementary, discrete and its limit set is identified canonically with the limit set of $\Gamma$.
\end{lemma}
\begin{proof}
$\Gamma^{ne}$ is nonelementary and discrete by definition.

By Proposition~\ref{P:NonelemElliptic} if $p_i(\Gamma)$ is not nonelementary then $p_i(\Gamma)$ is not a subgroup of $\PSL(2,\R)$ and it consists only of elliptic isometries with a common fixed point. Let $I$ be the set of all $i$ such that $p_i(\Gamma)$ is not nonelementary. Then for each representative geodesic $\gamma$ of each point in $\Lim_\Gamma$, the projection $\gamma_i:=p_i(\gamma)$ is constant for $i\in I$. The following mapping gives the identification of $\Lim_\Gamma$ and $\Lim_{\Gamma^{ne}}$:
$$ \gamma = (\gamma_1,\ldots,\gamma_{q+r})\mapsto (\gamma_{i_1},\ldots,\gamma_{i_n}).$$
\end{proof}

\begin{theorem}
\label{T:CharactLimitCone}
Let $\Delta$ be an irreducible arithmetic subgroup of $\PSL(2,\C)^q\times \PSL(2,\R)^r$ with $q+r\geq2$ and $\Gamma$ a finitely generated subgroup of $\Delta$ such that for at least one $j\in\{1,\ldots,q+r\}$, the projection $p_j(\Gamma)$ is nonelementary. Then the limit cone of $\Gamma$ consists of exactly one point if and only if $p_j(\Gamma)$ is contained in an arithmetic Fuchsian or Kleinian group.
\end{theorem}
\begin{proof}
The idea is to apply Theorem~\ref{T:MainCharactAlg} to $\Gamma^{ne}$. In order to do this we need that $n\geq 2$, because otherwise $\Gamma^{ne}$ is just a subgroup of $\PSL(2,\R)$ or $\PSL(2,\C)$.

If $n=1$, then the limit cone of $\Gamma$ contains only the line 
$$\left(\begin{bmatrix} 0 & 0 \\0 & 0\end{bmatrix}, \ldots, \begin{bmatrix} 0 & 0 \\0 & 0\end{bmatrix},\begin{bmatrix} t x_j & 0 \\0 & - t x_j\end{bmatrix},\begin{bmatrix} 0 & 0 \\0 & 0\end{bmatrix}, \ldots, \begin{bmatrix} 0 & 0 \\0 & 0\end{bmatrix}\right),\quad t \in \R.$$ 
But in this case, by Lemma~\ref{L:AtLeastTwoNonelem}, $p_j(\Gamma)$ is a subgroup of an arithmetic Fuchsian or Kleinian group.

If $n\geq 2$, the limit cone of $\Gamma$ is identified with the limit cone of $\Gamma^{ne}$, which coincides with $\overline{P_{\Gamma^{ne}}}$. By Theorem~\ref{T:MainCharactAlg}, then $\mathcal{L}_{\Gamma^{ne}}^{reg}$ is not empty and $P_{\Gamma^{ne}}$ consists of exactly one point if and only if $p_j(\Gamma)$ is contained in an arithmetic Fuchsian or Kleinian group for some $j\in\{1,\ldots,n\}$.
\end{proof}

\subsection{The structure of groups with an arithmetic projection and their limit set}
In Theorem~\ref{T:MainCharactAlg} we have seen that a finitely generated nonelementary subgroup $\Gamma$ of an irreducible arithmetic group in $\PSL(2,\C)^q\times\PSL(2,\R)^r$ has the smallest possible nonempty projective limit set only if $p_j(\Gamma)$ is a nonelementary subgroup of an arithmetic Fuchsian or Kleinian group for some $j\in{1,\ldots,q+r}$. In this section we determine the structure of $\Gamma$ and its limit set. Corollary~\ref{C:WidelyComm} finishes the proof of Theorem~B.

\begin{lemma}
\label{L:LimHomeom}
If $\Gamma_j:=p_j(\Gamma)$ is nonelementary and a subgroup of an arithmetic Fuchsian or Kleinian group, then $\Lim_{\Gamma_j}$ is homeomorphic to $\Lim_\Gamma$.
\end{lemma}
\begin{proof}
If $\Gamma_i:=p_i(\Gamma)$ is nonelementary only if $i=j$, then clearly $\Lim_{\Gamma_j}$ is homeomorphic to $\Lim_\Gamma$.

We assume that there is at least one more $i\neq j$ such that $\Gamma_i$ is nonelementary. Then by Lemma~\ref{L:GammaNE}, the group $\Gamma^{ne}$ is a nonelementary subgroup of $\PSL(2,\C)^{q'}\times\PSL(2,\R)^r$ with $q'\leq q$ and $q+r\geq 2$ and its limit set is identified with the limit set of $\Gamma$. For simplicity of the notation, we will assume that $\Gamma=\Gamma^{ne}$ and also that $j=1$. The group $\Gamma_1$ does not need necessarily to be a subgroup of $\PSL(2,\C)$ even if $q\neq0$.

By Theorem~\ref{T:MainCharactAlg}, the regular limit set $\mathcal{L}_\Gamma^{reg}$ is not empty and $P_\Gamma$ consists of exactly one point.
Since $\mathcal{L}_\Gamma^{reg}$ equals the product $F_\Gamma \times P_\Gamma$ (see \ref{T:LinkFP}), $\mathcal{L}_\Gamma^{reg}$ is homeomorphic to $F_\Gamma$ and so it is contained in the generalized torus $(\partial \bbH^3)^q\times(\partial \bbH^2)^r$.

Theorem 5.12 in \cite{gL02} says that if $\mathcal{L}_\Gamma^{reg}$ is not empty, then the attractive fixed points of the loxodromic isometries in $\Gamma$ are dense in $\mathcal{L}_\Gamma$. Hence $\mathcal{L}_\Gamma^{reg}$ is dense in $\mathcal{L}_\Gamma$ and so $\mathcal{L}_\Gamma^{reg} = \mathcal{L}_\Gamma$ because $\mathcal{L}_\Gamma^{reg}$ is contained in the compact (and hence closed) generalized torus $(\partial \bbH^3)^q\times(\partial \bbH^2)^r$.

We have $\Gamma = \left\{(g_1, \phi_2(g_1), \ldots, \phi_{q+r}(g_1)) \mid g_1 \in p_1(\Gamma)\right\}$ where $\phi_i$ is a surjective homomorphism between $p_1(\Gamma)$ and $p_i(\Gamma)$. We define the group $\Gamma_{1i} := \left\{(g_1, \phi_i(g_1)) \mid g_1 \in p_1(\Gamma)\right\}$.

If $p_1(\Gamma)$ is a subgroup of an arithmetic Fuchsian group, then for all $i=1,\ldots,q+r$, the group $\Gamma_i:=p_i(\Gamma)$ is also a subgroup of an arithmetic Fuchsian group and hence $\Tr(\Gamma_i)$ is a subset of $\R$. Then by Corollary~\ref{C:Zariski}, the Zariski closures over $\R$ of $\Gamma_i$ is (a conjugate of) $\PSL(2,\R)$. Since $P_\Gamma$ consists of exactly one point, $P_{\Gamma_{1i}}$ consists also of exactly one point for all $i = 1,\ldots,q+r$. Hence by Benoist's Theorem \ref{T:BenoistLimitCone}, the group $\Gamma_{1i}$ is not Zariski dense in (a conjugate of) $\PSL(2,\R) \times \PSL(2,\R)$ and then by the criterion of Dal'Bo and Kim (Theorem \ref{T:DalBoZariski}), for all $i = 1,\ldots,q+r$, the homomorphism $\phi_i$ can be extended to a continuous isomorphism $A_i$ between conjugates of $\PSL(2,\R)$, which, according to Theorem~\ref{T:PSLAutomorphisms}, is given by a conjugation with an element $A_i = \begin{bmatrix}a_i & b_i\\c_i & d_i\end{bmatrix} \in \GL(2,\R)$. Hence for all $g=(g_1,\ldots,g_{q+r})\in \Gamma$, $g_i= A_ig_1A_i^{-1}$.

If $\Gamma_1$ is a subgroup of an arithmetic Kleinian group but not a subgroup of an arithmetic Fuchsian group (this is possible only in the case $r=0$), then by Corollary~\ref{C:Zariski}, the Zariski closures over $\R$ of $\Gamma_i$ is $\PSL(2,\C)$. As above, for each $i=1,\ldots,q$, we can find $A_i = \begin{bmatrix}a_i & b_i\\c_i & d_i\end{bmatrix} \in \GL(2,\C)$ such that for all $g=(g_1,\ldots,g_q)\in \Gamma$, $g_i= A_ig_1A_i^{-1}$, or for all $g=(g_1,\ldots,g_q)\in \Gamma$, $g_i= A_i\overline{g_1}A_i^{-1}$, where $\overline{g_1}$ denotes the complex conjugation.

If $\xi$ is an attractive fixed point of an element $g_1 = p_1(g)$ in $\Gamma_1$ with $g =(g_1,\ldots,g_{q+r}) \in \Gamma$, then either $A_i(\xi):= \frac{a_i \xi + b_i}{c_i \xi + d_i}$ or $A_i(\xi):= \frac{a_i \bar{\xi} + b_i}{c_i \bar{\xi} + d_i}$ is the attractive fixed point of $g_i$ and vice versa. The maps $A_i$ are homeomorphisms of $\partial\bbH^3$ (and in the first case of $\partial\bbH^2$).

We consider the mapping $A:\Lim_{\Gamma_1} \rightarrow \Lim_{\Gamma}$, $z\mapsto(z, \widetilde{A_2}(z),\ldots,\widetilde{A_r}(z))\times (1:\ldots:1)$ where $\widetilde{A_i}(z):=\frac{a_i z + b_i}{c_i z + d_i}$ if $\sigma$ is the identity and $\widetilde{A_i}(z):=\frac{a_i \bar{z} + b_i}{c_i \bar{z} + d_i}$ if $\sigma$ is the complex conjugation, $i=1,\ldots,q+r$. This mapping is a homeomorphism on its image, i.e. $A: \Lim_{\Gamma_1} \rightarrow A(\Lim_{\Gamma_1})$ is a homeomorphism. Since $A$ is a bijection between the attractive fixed points of the loxodromic isometries in $\Gamma$ and the attractive fixed points of the loxodromic isometries in $\Gamma_1$, and since the attractive fixed points of the loxodromic isometries are dense in the corresponding limit set, $\mathcal{L}_\Gamma=A(\Lim_{\Gamma_1})$. Therefore $A: \Lim_{\Gamma_1} \rightarrow \Lim_{\Gamma}$ is a homeomorphism.
\end{proof}

The above proof is also the proof of the following corollary.

\begin{corollary}
\label{C:WidelyComm}
Let $\Gamma$ be a finitely generated nonelementary subgroup of an irreducible arithmetic group in $\PSL(2,\C)^q\times\PSL(2,\R)^r$ with $q+r\geq 2$. If $S:=p_j(\Gamma)$ is a subgroup of an arithmetic Fuchsian or Kleinian group for some $j\in{1,\ldots,q+r}$, then $\Gamma$ is a conjugate by an element in $\GL(2,\C)^q\times\GL(2,\R)^r$ of a group
$$ \Diag(S):=\{(\sigma_1(s),\ldots,\sigma_{q+r}(s))\mid s\in S\},$$
where, for $i=1,\ldots,q+r$, $\sigma_i$ denotes either the identity or the complex conjugation.
\end{corollary}

This corollary and Theorem~\ref{T:MainCharactAlg} prove Theorem~B from the introduction.

\subsection{Small limit sets}

\subsubsection{Subgroups of $\PSL(2,\C)^q\times\PSL(2,\R)^r$}
In this section we answer the question when the limit set of $\Gamma$ is topologically a circle or a subspace of a circle where $\Gamma$ is a finitely generated subgroup of an irreducible arithmetic group in $\PSL(2,\C)^q\times \PSL(2,\R)^r$ with $q+r\geq 2$. 

We say that a set $X$ is embedded homeomorphically in a circle if there exists a map $f: X \rightarrow S^1$ such that  $f: X \rightarrow f(X)$ is a homeomorphism.

\begin{theorem}
\label{T:TopolCharactAlgAlg}
Let $\Gamma$ be a finitely generated subgroup of an irreducible arithmetic group in $\PSL(2,\C)^q\times\PSL(2,\R)^r$ with $q + r \geq 2$ and $r\neq 0$ such that $p_{j}(\Gamma)$ is nonelementary for some $j\in\{1,\ldots,q+r\}$. Then $\mathcal{L}_\Gamma$ is embedded homeomorphically in a circle if and only if $p_{j}(\Gamma)$ is contained in an arithmetic Fuchsian group.
\end{theorem}
\begin{proof}
If $\Gamma_j:=p_j(\Gamma)$ is a subgroup of an arithmetic Fuchsian group, then $\Lim_\Gamma$ and $\Lim_{\Gamma_j}$ are homeomorphic (Lemma~\ref{L:LimHomeom}). Since $\Lim_{\Gamma_j}$ is a topological subspace of $S^1$, the limit set $\mathcal{L}_\Gamma$ is embedded homeomorphically in a circle.

Now let $p_j(\Gamma)$ be such that it is not contained in an arithmetic Fuchsian group. By Lemma~\ref{L:AtLeastTwoNonelem} and Lemma~\ref{L:GammaNE}, the group $\Gamma^{ne}$ is a nonelementary subgroup of $\PSL(2,\C)^{q'}\times\PSL(2,\R)^r$ with $q'\leq q$ and $q+r\geq 2$ and its limit set is identified with the limit set of $\Gamma$. For simplicity of the notation, we will assume that $\Gamma=\Gamma^{ne}$.

Then by Theorem~\ref{T:MainCharactAlg}, $P_\Gamma$ contains at least two different points and by Lemma~\ref{L:Connected} there is a path in $P_\Gamma$ between these points and thus $P_\Gamma$ contains an interval. Let $I$ be an open subinterval (contained in this interval.)

The next step is to show that $F_\Gamma$ is infinite. Since $\mathcal{L}_\Gamma$ is nonempty, there is at least one loxodromic element $\tilde{h}$ in $\Gamma$. By Lemma~\ref{L:Schottky}, starting from $\tilde{h}$ and $\tilde{h}$, we can find loxodromic isometries $g$ and $h$ in $\Gamma$ such that the groups generated by the corresponding components of $g$ and $h$ are Schottky with only loxodromic elements. The projections of the attractive fixed points of $g^khg^{-k}$, $k \in \N$, in the Furstenberg boundary give us infinitely many points in $F_\Gamma$.

Consequently, since $F_\Gamma$ is closed and lies in the generalized torus $(\partial \bbH^3)^q\times(\partial \bbH^2)^r$, it contains a point $\xi$ that is not isolated. This means that any neighborhood $U$ of $\xi$ in $F_\Gamma$ contains a point $\xi_U$ different from $\xi$.

Let us assume that there is a topological embedding $f:\mathcal{L}_\Gamma \rightarrow S^1$, i.e $f$ is a homeomorphism between $\mathcal{L}_\Gamma$ and $f(\mathcal{L}_\Gamma)$ with the subset topology. Since $\{\xi\}\times I \subset \mathcal{L}_\Gamma$ is connected, the image $f(\{\xi\}\times I)$ is also connected. Hence $f(\{\xi\}\times I)$ is an arc in $S^1$ and thus open in $S^1$ and so in $f(\mathcal{L}_\Gamma)$. Since $f$ is homeomorphism, the preimage of the open set $f(\{\xi\}\times I)$ is open in $\mathcal{L}_\Gamma$, i.e. $\{\xi\}\times I$ is open in $\mathcal{L}_\Gamma$. 

The topology of $((\partial \bbH^3)^q\times(\partial \bbH^2)^r)_{reg}$ is the product topology and $\mathcal{L}_\Gamma^{reg}\subseteq((\partial \bbH^3)^q\times(\partial \bbH^2)^r)_{reg}$ has the product subspace topology. In particular, each open set $V$ containing $\{\xi\}\times I$ contains also $\{\xi_U\}\times I$ where $U$ is a neighborhood of $\xi$ contained in the projection of $V$ in $F_\Gamma$. Hence, since any neighborhood $U$ of $\xi$ in $F_\Gamma$ contains a point $\xi_U$ different from $\xi$,  $\{\xi\}\times I$ is not open in $\mathcal{L}_\Gamma$. This is a contradiction.
\end{proof}

Theorem~\ref{T:TopolCharactAlgAlg} allows us to decide whether $p_j(\Gamma)$ is a subgroup of an arithmetic Fuchsian group or not. The next corollary distinguishes when $p_j(\Gamma)$ is a (cofinite) arithmetic Fuchsian group and when it is not.

\begin{corollary}
\label{C:TopolCharactAlgAlg}
Let $\Gamma$ be a finitely generated subgroup of an irreducible arithmetic group in $\PSL(2,\C)^q\times\PSL(2,\R)^r$ with $q + r \geq2$ and $r\neq 0$ such that $p_{j}(\Gamma)$ is nonelementary for some $j\in\{1,\ldots,q+r\}$. Then $\mathcal{L}_\Gamma$ is homeomorphic to a circle if and only if $p_{j}(\Gamma)$ is a cofinite arithmetic Fuchsian group.
\end{corollary}
\begin{proof}
If $\Gamma_j:=p_{j}(\Gamma)$ is not a subgroup of an arithmetic Fuchsian group, then by the previous theorem, $\Lim_\Gamma$ is not homeomorphic to a circle.

Now let $\Gamma_j$ be a subgroup of an arithmetic Fuchsian group. Then by Lemma~\ref{L:LimHomeom}, $\Lim_\Gamma$ is homeomorphic to $\Lim_{\Gamma_j}$.

If $\Lim_\Gamma$ be homeomorphic to a circle, then $\Lim_{\Gamma_j}$ is connected. Since $\Gamma_j$ is nonelementary, $\Lim_{\Gamma_j}$ contains more than two points and by Theorem 3.4.6 in \cite{sK92}, it is either the whole boundary $\partial\bbH^2$ of $\bbH^2$ or it is nowhere dense in $\partial\bbH^2$ and in particular not connected. Hence $\Lim_{\Gamma_j}$ is $\partial\bbH^2$. 

Combining Theorem 4.6.1 and Theorem 4.5.1 in \cite{sK92} we get that a finitely generated Fuchsian group of the first kind, i.e. whose limit set is $\partial\bbH^2$, has a fundamental region of finite hyperbolic area. Therefore $\Gamma_j$ is cofinite arithmetic Fuchsian group.

The converse is also true, namely, if $\Gamma_j$ is cofinite arithmetic Fuchsian group, then $\Lim_{\Gamma_j}$ and hence $\Lim_\Gamma$ are homeomorphic to $S^1$.
\end{proof}

\subsubsection{Subgroups of $\PSL(2,\C)^q$}
We consider a subgroup $\Gamma$ of an irreducible arithmetic group in $\PSL(2,\C)^q$. The question we answer is when $\Lim_\Gamma$ is a sphere. The next theorem is analogous to Corollary~\ref{C:TopolCharactAlgAlg} but it needs an additional structure on the geometric boundary $\partial(\bbH^3)^q$.

The geometric boundary $\partial(\bbH^3)^q$ is homeomorphic to the unit tangent sphere at a point in $(\bbH^3)^q$. This unit tangent sphere has a natural smooth structure induced by the Riemannian metric of $(\bbH^3)^q$ and this makes the unit tangent sphere diffeomorphic to the standard $(3q-1)$-sphere and defines a smooth structure on $\partial(\bbH^3)^q$.
 
\begin{theorem}
\label{T:TopolCharactC}
Let $\Gamma$ be a subgroup of an irreducible arithmetic group in $\PSL(2,\C)^q$ with $q \geq 2$ such that $p_{j}(\Gamma)$ is a cofinite Kleinian group for some $j\in\{1,\ldots,q\}$. Then $\mathcal{L}_\Gamma$ is the image of a differentiable embedding of the 2-sphere $S^2$ in $\partial(\bbH^3)^q$ if and only if $p_{j}(\Gamma)$ is an arithmetic Kleinian group.
\end{theorem}
\begin{proof}
For simplicity we will denote $p_i(\Gamma)$ by $\Gamma_i$ for all $i = 1,\ldots,q$. Without loss of generality, we can assume that $p_{j}(\Gamma)$ is a cofinite Kleinian group for $j=1$.

In order to show the first implication let $\Gamma_1$ be an arithmetic Kleinian group. Then $\Lim_\Gamma$ and $\Lim_{\Gamma_1}$ are homeomorphic (Lemma~\ref{L:LimHomeom}). From the proof of Lemma~\ref{L:LimHomeom} it follows that $\Lim_\Gamma$ and $\Lim_{\Gamma_1}$ are diffeomorphic because $ \widetilde{A_i}$ are diffeomorphisms of the Riemann sphere for $i=1,\ldots,q$.

Thus $\Lim_\Gamma$ is the image of an embedding of the $S^2$ in $\partial(\bbH^3)^q$.

In order to show the second implication, we assume that $\Gamma_1$ is not arithmetic and the limit set of $\Gamma$ is the image of an embedding of the sphere $S^2$ in $\partial(\bbH^3)^q$. By Lemma~\ref{L:GammaNE}, $\Lim_\Gamma$ and $\Lim_{\Gamma^{ne}}$ are diffeomorphic. Here $\partial(\bbH^3)^{q'}$ is a submanifold of $\partial(\bbH^3)^q$.

The projective limit set $P_{\Gamma^{ne}}$ contains at least two different points (Theorem~\ref{T:MainCharactAlg}) and by Theorem~\ref{T:ConvexAlg}, it contains also a path joining them.

Theorem 4.10 in \cite{gL06} says that the set of attractive fixed points of loxodromic isometries in a nonelementary subgroup of $\PSL(2,\C)^n$ is dense in its limit set. It follows that $\Lim_{\Gamma^{ne}}^{reg}$ is dense in $\Lim_{\Gamma^{ne}}$. 
Additionally, since $\partial(\bbH^3)^n_{reg}$ is open in $\partial(\bbH^3)^n$, $\Lim_{\Gamma^{ne}}^{reg}$ is open in $\Lim_{\Gamma^{ne}}$. Hence $\Lim_{\Gamma^{ne}}^{reg}$ is open and dense in $\Lim_{\Gamma^{ne}}$

The limit set $\Lim_{\Gamma_1}$ of $\Gamma_1$ is $\partial\bbH^3$, i.e. homeomorphic to $S^2$. To each point in $\Lim_{\Gamma_1}$ corresponds at least one point in $\Lim_{\Gamma^{ne}}$. We will show that each point in $\Lim_{\Gamma_1}$ is the projection in the first factor of a point in $\Lim_{\Gamma^{ne}}^{reg}$.

For each point in $\Lim_{\Gamma^{ne}}^{sing}$ there is a sequence in $\Lim_{\Gamma^{ne}}^{reg}$ that converges to it. Let $\xi$ be a point in $\Lim_{\Gamma^{ne}}^{sing}$ such that the projection of one of its representative geodesics in the first factor is not a constant. This means that the first factors of the regular elements converge to the first factor of $\xi$. The projections of the regular elements in the Furstenberg boundary  $(\partial\bbH^3)^q$ have an accumulation point in $(\partial\bbH^3)^q$ (because  $(\partial\bbH^3)^q$ is compact). Since $F_{\Gamma^{ne}}$ is closed, there is a point in $\Lim_{\Gamma^{ne}}^{reg}$ such that its projection in the first factor coincides with the first factor of $\xi$.

Hence the projection of $\Lim_{\Gamma^{ne}}^{reg}$ on the first factor of the regular boundary is $\Lim_{\Gamma_1}$, which is the 2-sphere. Since $\Lim_{\Gamma^{ne}}^{reg}=F_{\Gamma^{ne}}\times P_{\Gamma^{ne}}$ by Theorem~\ref{T:LinkFP}, the projection of $\Lim_{\Gamma^{ne}}^{reg}$ on the first factor times $P_{\Gamma^{ne}}$ contains a set homeomorphic to $\R^3$. On the other side $\Lim_{\Gamma^{ne}}^{reg}$ is a two dimensional smooth submanifold of $\partial(\bbH^3)_{reg}$ and therefore its projection on the first factor times $P_{\Gamma^{ne}}$ is two dimensional, which is impossible.


Thus if $\Gamma_1$ is not arithmetic, then the limit set of $\Gamma^{ne}$ is not the image of an embedding of the $S^2$ in $\partial(\bbH^3)^q$.
\end{proof}
\begin{rem}
The fact that the projection of $\Lim_{\Gamma^{ne}}^{reg}$ on the first factor times $P_{\Gamma^{ne}}$ contains a set homeomorphic to $\R^3$ while $\Lim_{\Gamma^{ne}}^{reg}$ is two dimensional is not a contradiction if $\Lim_{\Gamma^{ne}}^{reg}$ is just assumed to be homeomorphic to $S^2$.
\end{rem}

We can give a more precise answer to the question when the limit set is topologically a circle. An answer is given by the next theorem. We use the following definition. A \textit{quasi-Fuchsian} group is a subgroup of $\PSL(2,\C)$ whose limit set is homeomorphic to a circle.
\begin{theorem}
\label{T:TopolCharactCFuchs}
Let $\Gamma$ be a finitely generated subgroup of an irreducible arithmetic group in $\PSL(2,\C)^q$ with $q \geq 2$ such that $p_{j}(\Gamma)$ is nonelementary for some $j\in\{1,\ldots,q\}$. Then $\mathcal{L}_\Gamma$ is homeomorphic to a circle if and only if $p_{j}(\Gamma)$ is a cofinite arithmetic Fuchsian group or a quasi-Fuchsian subgroup of an arithmetic Kleinian group.
\end{theorem}
\begin{proof}
Again for simplicity we will denote $p_i(\Gamma)$ by $\Gamma_i$ for all $i = 1,\ldots,q$. Without loss of generality, we can assume that $p_{j}(\Gamma)$ is nonelementary for $j=1$.

We will consider two main cases: when $\Gamma_1$ is a subgroup of an arithmetic Fuchsian or Kleinian group and when it is not.

Let $\Gamma_1$ be an arithmetic Fuchsian or Kleinian group, then $\Lim_\Gamma$ and $\Lim_{\Gamma_1}$ are homeomorphic (Lemma~\ref{L:LimHomeom}). In the next three paragraphs we consider the different possibilities for $\Gamma_1$.

If $\Gamma_1$ is a cofinite arithmetic Fuchsian group or a quasi-Fuchsian subgroup of an arithmetic Kleinian group, then $\Lim_{\Gamma_1}$ is topologically a circle.

If $\Gamma_1$ is a subgroup of an arithmetic Fuchsian group but is not cofinite, then by Theorem 4.6.1 and Theorem 4.5.1 in Katok's book \cite{sK92}, $\Gamma_1$ is not of the first kind and since it is nonelementary from Theorem 3.4.6 in \cite{sK92} follows that $\Lim_{\Gamma_1}$ is nowhere dense in the circle in $\partial\bbH^3$ that is left invariant by $\Gamma_1$ and in particular not connected. Thus in this case $\Lim_\Gamma$ is not homeomorphic to a circle.

Let $\Gamma_1$ be a subgroup of an arithmetic Kleinian group. If $\Gamma_1$ is (conjugated to) a Fuchsian group, then it is contained in an arithmetic Fuchsian group and we are in the previous case. If $\Gamma_1$ contains purely loxodromic elements, then $\Lim_{\Gamma_1}$ is homeomorphic to a circle if and only if $\Gamma_1$ is a quasi-Fuchsian group.

Now we come back to the second big case. Let $\Gamma_1$ be such that it is not contained in an arithmetic Fuchsian or Kleinian group. By Lemma~\ref{L:GammaNE}, $\Lim_\Gamma$ and $\Lim_{\Gamma^{ne}}$ are homeomorphic, $\Gamma^{ne}$ is a subgroup of $\PSL(2,\C)^n$ and by Lemma~\ref{L:AtLeastTwoNonelem} $n\geq 2$. 

Then according to Theorem \ref{T:MainCharactAlg}, $P_{\Gamma^{ne}}$ contains at least two different points and by Lemma~\ref{L:NonemptyAlg} there is a path in $P_{\Gamma^{ne}}$ between these points and thus $P_{\Gamma^{ne}}$ contains an interval. The rest of the proof in this case is analogous to the corresponding part in Theorem~\ref{T:TopolCharactAlgAlg} that shows that $\Lim_\Gamma$ is not homeomorphically embedded in a circle.
\end{proof}
\begin{rem}
From the proof it follows in particular that if $\Lim_\Gamma$ is not contained in a circle, then $p_j(\Gamma)$ is not a subgroup of an arithmetic Fuchsian or Kleinian group.
\end{rem}

\subsection{Totally geodesic embeddings}
\setcounter{theorem}{0}
The following theorem is a corollary of Theorem~\ref{T:TopolCharactAlgAlg} and Theorem~\ref{T:TopolCharactCFuchs}. A proof can be found in \cite{sG09}.

\begin{theorem}
\label{T:TotallyGeodesicAlg}
Let $\Gamma$ be a finitely generated subgroup of an irreducible arithmetic group in $\PSL(2,\C)^q\times\PSL(2,\R)^r$ with $q + r \geq 2$ such that $p_{j}(\Gamma)$ is nonelementary for some $j\in\{1,\ldots,q+r\}$. Then there is a totally geodesic embedding of $\bbH^2$ in $(\bbH^3)^q\times(\bbH^2)^r$ that is left invariant by the action of $\Gamma$ if and only if  $p_j(\Gamma)$ is a subgroup of an arithmetic Fuchsian group.
\end{theorem}

\end{document}